\documentclass{article}

\usepackage{arxiv}

\usepackage[utf8]{inputenc} 
\usepackage[T1]{fontenc}    
\usepackage[hypertexnames=true]{hyperref}       
\usepackage{url}            
\usepackage{booktabs}       
\usepackage{amsmath, amsfonts}       
\usepackage{amsthm}
\usepackage{nicefrac}       
\usepackage{microtype}      
\usepackage{cleveref}       
\usepackage{graphicx}
\usepackage{doi}

\usepackage{amssymb}
\usepackage{algorithmic}
\usepackage[ruled,linesnumbered,vlined]{algorithm2e}
\usepackage{array}
\usepackage{textcomp}
\usepackage{stfloats}
\usepackage{verbatim}
\usepackage{cite}
\usepackage{multirow} 
\usepackage{scalerel}
\usepackage{tikz}
\usepackage[title,titletoc]{appendix}
\usepackage{multirow}
\usepackage{amssymb}
\usepackage{pifont}
\usepackage{caption}
\usepackage{algorithmic}
\usepackage[ruled,linesnumbered,vlined]{algorithm2e}
\usepackage{hyperref}
\usepackage{multirow}
\usepackage{subcaption}
\newdimen\figrasterwd
\figrasterwd\textwidth
\usepackage[shortlabels]{enumitem}

\theoremstyle{plain}

\newtheorem{lemma}{Lemma}
\newtheorem{proposition}{Proposition}
\theoremstyle{definition}

\newtheorem{definition}{Definition}
\theoremstyle{remark}
\newtheorem{remark}{Remark}

\title{Combining Gradient Information and Primitive Directions for High-Performance Mixed-Integer Optimization}

\author{
	\hspace{1mm}Matteo Lapucci \\
	Global Optimization Laboratory (GOL) \\
	Dip.\ di Ingegneria dell'Informazione \\
	Università di Firenze \\
	Via di S.\ Marta, 3, 50135, Firenze, Italy \\
	\texttt{matteo.lapucci@unifi.it} \\
	\And
	\hspace{1mm}Giampaolo Liuzzi \\
	Dip.\ di Ingegneria Informatica, \\
	Automatica e Gestionale ``Antonio Ruberti''\\
	Università di Roma ``La Sapienza'' \\
	Via Ariosto, 25, 00185, Roma, Italy \\
	\texttt{liuzzi@diag.uniroma1.it} \\
	\And
	\hspace{1mm}Stefano Lucidi \\
	Dip.\ di Ingegneria Informatica, \\
	Automatica e Gestionale ``Antonio Ruberti''\\
	Università di Roma ``La Sapienza'' \\
	Via Ariosto, 25, 00185, Roma, Italy \\
	\texttt{lucidi@diag.uniroma1.it} \\
	\And
	\hspace{1mm}Pierluigi Mansueto \\
	Global Optimization Laboratory (GOL) \\
	Dip.\ di Ingegneria dell'Informazione \\
	Università di Firenze \\
	Via di S.\ Marta, 3, 50135, Firenze, Italy \\
	\texttt{pierluigi.mansueto@unifi.it} \\
}

\renewcommand{\headeright}{Matteo Lapucci, Giampaolo Liuzzi, Stefano Lucidi and Pierluigi Mansueto}
\renewcommand{\undertitle}{}
\renewcommand{\shorttitle}{}

\hypersetup{
pdftitle={Combining Gradient Information and Primitive Directions for High-Performance Mixed-Integer Optimization},
pdfsubject={},
pdfauthor={Matteo Lapucci, Giampaolo Liuzzi, Stefano Lucidi, Pierluigi Mansueto},
pdfkeywords={Mixed-integer nonlinear optimization, Gradient-based optimization, Primitive directions, Global convergence},
}

\begin{document}
\maketitle

\begin{abstract}
	In this paper we consider bound-constrained mixed-integer optimization problems where the objective function is differentiable w.r.t.\ the continuous variables for every configuration of the integer variables. We mainly suggest to exploit derivative information when possible in these scenarios: concretely, we propose an algorithmic framework that carries out local optimization steps, alternating searches along gradient-based and primitive directions. The algorithm is shown to match the convergence properties of a derivative-free counterpart. Most importantly, the results of thorough computational experiments show that the proposed method clearly outperforms not only the derivative-free approach but also the main alternatives available from the literature to be used in the considered setting, both in terms of efficiency and effectiveness.
\end{abstract}

\MSCs{90C11 \and 90C26 \and 90C30}

\section{Introduction}
\label{sec:introduction}
In this paper, we deal with nonlinear mixed-integer optimization problems of the form
\begin{equation}
	\label{eq:prob}
	\begin{aligned}
		\min_{x\in\mathbb{R}^n, z\in\mathbb{Z}^m}\;&f(x,z)\\
		\text{s.t. }&l_x\le x\le u_x\\
		&l_z\le z\le u_z,
	\end{aligned}
\end{equation}
where $l_x,u_x\in\mathbb{R}^n$, with $l_x\le u_x$, are bounds on the feasible values of the continuous variables $x$ and $l_z,u_z\in\mathbb{Z}^m$, with $l_z\le u_z$, are the bounds for the integer variables $z$. We denote by $\Omega_x$ and $\Omega_z$ the sets $\{x\in\mathbb{R}^n\mid l_x\le x\le u_x\}$ and $\{z\in\mathbb{Z}^m\mid l_z\le z\le u_z\}$, respectively, and $\Omega = \Omega_x\times \Omega_z$. Further, we assume that the function $f:\mathbb{R}^n\times\mathbb{R}^m\to \mathbb{R}$ is continuously differentiable w.r.t.\ the continuous variables, i.e., $\nabla_xf(x,z)$ exists and is continuous over $\Omega_x$ for any $z\in\Omega_z$. In what follows, the Euclidean norm in $\mathbb{R}^n$ will be denoted as $\|\cdot\|$, whereas we denote by $\Pi_C$ the euclidean projection operator onto the convex set $C\subseteq \mathbb{R}^n$. Finally, we denote by $\mathbf{1}_N$ and $\mathbf{0}_N$, with $N \in \mathbb{N}$, the $N$-dimensional vectors of all ones and
all zeros, respectively.

Real world problems presenting this structure are not uncommon. For instance, we can consider design problems where the values of integer variables induce the overall structure of some object or building. Once the structure is fixed, geometry and material properties can be varied with continuity, according to some well-known law, to obtain an overall result with the desired property \cite{liuzzi2012decomposition,Arroyo2009,Sobieszczanski-Sobieski1997}. Other examples of applications can be found for generalized Nash equilibrium \cite{Facchinei10} and diffusion processes statistical inference problems \cite{Pedersen01}.

Most algorithmic approaches from the literature to deal with nonlinear mixed-integer problems work under the assumption that the objective function is computed by a black-box oracle whose internal mechanisms are inaccessible \cite{giovannelli2022derivative,le2023taxonomy,Audet2017,BOUKOUVALA2016701,conn2009introduction,larson2019derivative}. However, as noted above, this is not necessarily the case; indeed, some frameworks do exploit this information when possible, see, e.g., \cite{BONAMI2008186, Exler2007}. In this paper, we underline that the employment of first-order information shall really be fostered, as it can lead to substantial improvements in both efficiency and effectiveness of black-box algorithmic frameworks. 

Here, we put particular emphasis on the \textit{Nonsmooth Derivative-Free Linesearch} (\texttt{DFNDFL}) framework from \cite{giovannelli2022derivative}, which combines:
\begin{itemize}
	\item local searches, based on derivative-free line searches along both special unitary directions and the coordinate directions, to carry out an exploration with respect to the continuous variables;
	\item local searches, along so-called primitive directions, to perform an optimization with respect to the integer variables.
\end{itemize}
The line search carried out along the dense direction in \texttt{DFNDFL} conveys properties on the initial point, whereas the discrete search conveys properties on the intermediate point where only continuous variables have been moved. The convergence of particular subsequences is then obtained by exploiting the fact that the steps used in the continuous search vanish in the limit.  Then, limit properties can be showed for the sequence $\{(x_{k},z_k)\}$.

According to the above principles, we could think of employing gradient-based line searches for the continuous part of the problem, when derivatives are available. In fact, from the one hand we underline that this change leads to no loss of convergence guarantees - although the analysis needs some nontrivial adjustments; on the other hand, we show by thorough computational experiments that the resulting algorithm exhibits much better performance than the original derivative-free counterpart.

We shall note that recently, in \cite{cristofari2024mixalm}, a related approach 
has been proposed for mixed-integer problems with general nonlinear constraints. The algorithm follows a similar scheme: it proceeds by alternating a first step of optimization of the continuous variables and a second step of exploration of the discrete ones. In particular, with respect to the continuous variables, the (approximate) minimization of the augmented Lagrangian function is required to be carried out at each iteration. This strategy allows to assess properties on the point obtained by moving the continuous variables. Then, limit properties can be proved for the sequence $\{(x_{k+1},z_k)\}$. 
A strategy of this kind could naturally be adopted also in the present unconstrained setting. However, in this paper we focus on a strategy more closely resembling that of \texttt{DFNDFL}: a single (derivative-based) line search is required to update the continuous variables, rather than a full minimization.  

{One of the key benefits of \texttt{DFL}-type algorithmic framework consists in the objective function being evaluated only at points with integer values for $z$ variables: no relaxation of the integrality constraint is ever required. This is crucial for all those tasks where integer variables $z$ denote physically integer quantity that are unrelaxable, i.e., measures of $f$ do not exist at all if non-integer values are provided in input for variables $z$.} 
{Among the methods that enjoy this property, the only ones, to the best of our knowledge, that exploit gradient information for continuous variables are the Mixed-Integer SQP approaches (\texttt{MISQP}) \cite{Exler2007}; however, these methods do not alternate local optimization steps w.r.t.\ $x$ and $z$ variables, but rather carry out joint optimization steps employing rough approximations of derivatives for integer variables. The resulting algorithm, however, does not possess any theoretical convergence guarantee.}

The rest of the paper is organized as follows. In Section \ref{sec:prelim}, we report some preliminary definitions and properties and we recall the \texttt{DFNDFL} algorithm which inspired the present work. The proposed algorithm to solve problem \eqref{eq:prob} is presented in Section \ref{sec:algo}. Then, in 
Section \ref{sec::convergence}, we analyze the convergence properties of the algorithm. 
Section \ref{sec::exp} is devoted to the numerical experiments. Finally, in Section \ref{sec:conclusion} we draw some conclusions and discuss possible lines of future research. Appendix \ref{app:L-BFGS-B} contains the formal proof that first step of the \texttt{L-BFGS-B} algorithm \cite{byrd95} is performed moving along a gradient-related direction, which is important to assess the convergence guarantees of the implemented version of the algorithm. 

\section{Preliminaries}
\label{sec:prelim}
The algorithmic framework considered in this manuscript makes use of directions that are completely zero in the continuous or in the integer part; in other words, variables update concern continuous variables only or integer variables only. 

Now, updates of integer variables exploit particular search directions, which we formally characterize hereafter.
\begin{definition}[{\cite[Def.\ 4]{liuzzi2020algorithmic}}]
	A \textit{vector} $v\in\mathbb{Z}^m$ is called \textit{primitive} if the greatest common divisor of its components is equal to 1.
\end{definition}
Local searches w.r.t.\ integer variables will be  carried out checking steps along some primitive vectors. This concept needs the following definition to be formalized.
\begin{definition}[{\cite[Def.\ 4]{giovannelli2022derivative}}]
	Given a point $(x,z)\in\Omega$, the \textit{set of feasible primitive directions} at $(x, z)$ with respect to $\Omega$ is given by
	$$\mathcal{D}^z(x,z) = \{d\in\mathbb{Z}^m\mid\ d \text{ is a primitive vector and }  (x,z+d)\in\Omega\}.$$
\end{definition}

We then provide some additional definitions in order to characterize local optimality and stationarity for problem \eqref{eq:prob}.
\begin{definition}[{\cite[Def.\ 5]{giovannelli2022derivative}}]
	Given a point $(\bar{x},\bar{z})\in\Omega$, the \textit{discrete neighborhood} of $(\bar{x},\bar{z})$ is defined as
	$$\mathcal{B}^z(\bar{x},\bar{z}) = \{(\bar{x},z)\in\Omega\mid z = \bar{z}+d,\;d\in\mathcal{D}^z(\bar{x}, \bar{z})\}.$$
\end{definition}
\begin{definition}[{\cite[Def.\ 6]{giovannelli2022derivative}}]
	Given a point $(\bar{x},\bar{z})\in\Omega$ and a scalar $\tau$, the continuous neighborhood of $(\bar{x},\bar{z})$ is given by 
	$$\mathcal{B}^x(\bar{x},\bar{z},\tau) =\{(x,\bar{z})\in\Omega\mid \|x-\bar{x}\|\le \tau\}.$$
\end{definition}
\begin{definition}[{\cite[Def.\ 7]{giovannelli2022derivative}}]
	A point $(\bar{x},\bar{z})\in\Omega$ is a \textit{local minimum point} of problem \eqref{eq:prob} if, for some $\epsilon>0$, 
	\begin{align*}
		f(\bar{x},\bar{z})\le f(x,\bar{z})\quad&\text{for all } (x,\bar{z})\in\mathcal{B}^x(\bar{x},\bar{z},\epsilon);\\
		f(\bar{x},\bar{z})\le f(\bar{x},z)\quad&\text{for all } (\bar{x},z)\in\mathcal{B}^z(\bar{x},\bar{z}).
	\end{align*}
\end{definition}

\begin{definition}
	\label{def::stat}
	A point $(\bar{x},\bar{z})\in\Omega$ is a stationary point of problem \eqref{eq:prob} if
	\begin{align*}
		&\nabla_x f(\bar{x},\bar{z})^\top(x - \bar{x}) \ge 0,\quad\text{for all } x \in \Omega_x;\\
		&f(\bar{x},\bar{z})\le f(\bar{x},z)\quad\qquad\quad\,\text{for all } (\bar{x},z)\in\mathcal{B}^z(\bar{x},\bar{z}).
	\end{align*}
\end{definition}

It is easy to observe that stationarity is actually a necessary condition for local optimality.
In the following, we also call a point $(\bar{x}, \bar{z}) \in \Omega$ satisfying the first condition of Definition \ref{def::stat} as \textit{stationary w.r.t.\ the continuous variables}.

Finally, we recall the notion of convergence to a point reported in \cite{Lapucci2023} for the mixed-integer scenario.
\begin{definition}[{\cite[Def.\ 3.1]{Lapucci2023}}]
	\label{def::convergence}
	A sequence $\{(x_k , y_k)\}$ converges to a point $(\bar{x}, \bar{y})$ if for any $\epsilon > 0$ there exists an index $k_\epsilon$ such that for all $k \ge k_\epsilon$ we have that $y_k = \bar{y}$ and $\|x_k - \bar{x}\| < \epsilon$.
\end{definition}

The \texttt{DFNDFL} approach from \cite{giovannelli2022derivative} is formally described in Algorithm \ref{alg::dfndfl}. The method consists of an alternate minimization w.r.t.\ continuous and discrete variables and, hence, is divided into two main phases.
At each iteration $k$ of the algorithm, minimization w.r.t.\ the continuous variables is first carried out starting at the current iterate $(x_k, z_k) \in \Omega$ (steps \ref{line::startpcs}-\ref{line::endpcs}). Then, the updated solution, named $(x_k^c, z_k)$, is used as starting point for the \textit{Discrete Search} (\texttt{DS}) phase (step \ref{line::ds}), where a set of primitive directions $D_k \subseteq \mathcal{D}^z(x_k^c, z_k)$ is investigated. Step \ref{line::ref} finally allows to define the next iterate $(x_{k+1}, z_{k + 1})$ as any solution improving $(x_k^c, z_k^d)$, making the framework quite flexible.

\begin{algorithm}[h]
	\caption{\texttt{DFNDFL}}
	\label{alg::dfndfl}
	\textbf{Input:}  $(x_0, z_0) \in \Omega$, $\xi_0 > 0$, $\delta \in (0, 1)$, $\gamma > 0$, $\eta \in (0, 1)$\\
	Choose $D_0 =\{v_1,\ldots,v_{|D_0|}\} \subseteq \mathcal{D}^z(x_0, z_0)$\\
	Let $\alpha_0^{D_0} = [\alpha_0^{v_1},\ldots,\alpha_0^{v_{|D_0|}}] = \boldsymbol{1}_{|D_0|}$ \\
	Let $\alpha_0^c = 1$ \\
	\For{$k = 0, 1,\ldots$}{
		Let $v_k^c \in \{v \in \mathbb{R}^n \mid \|v\| = 1\}$\label{line::startpcs}\\
		\If{$\exists \rho \in \{-1, 1\}$: $f\left(\Pi_{\Omega_x}\left(x_k + \rho\alpha_k^cv_k^c\right), z_k\right) \le f(x_k, z_k) - \gamma(\alpha_k^c)^2$}{
			$\alpha_{k + 1}^c = \max\limits_{h \in \mathbb{N}}\{\delta^{-h}\alpha_k^c\mid f\left(\Pi_{\Omega_x}\left(x_k + \rho\delta^{-h}\alpha_k^cv_k^c\right), z_k\right) \le f(x_k, z_k) - \gamma(\delta^{-h}\alpha_k^c)^2\}$\label{line::expansion}\\
			$x_k^c = \Pi_{\Omega_x}\left(x_k + \rho\alpha_{k + 1}^c v_k^c\right)$
		}
		\Else{
			$\alpha_{k + 1}^c = \eta\alpha_k^c$\label{line::reduction}\\
			$x_k^c = x_k$\label{line::endpcs}
		}
		$z_k^d$, $D_{k + 1}$, $\alpha_{k+1}^{D_{k+1}}$, $\xi_{k+1}$ = \texttt{DS}$\left(f, \Omega, (x_k^c, z_k), D_k, \alpha_k^{D_k}, \xi_k, \eta\right)$\label{line::ds}\\
		Find $(x_{k+1}, z_{k+1}) \in \Omega$ s.t. $f(x_{k+1}, z_{k+1}) \le f(x_k^c, z_k^d)$\label{line::ref}
	}
\end{algorithm}

\begin{algorithm}[!h]
	\caption{Discrete Search (\texttt{DS})}
	\label{alg::ds}
	\textbf{Input:} $(x, z) \in \Omega$, $D_k \subseteq \mathcal{D}^z(x, z)$, $\alpha_k^{D_k} \in \mathbb{R}^{|D_k|}$, $\xi_k > 0$, $\eta \in (0, 1)$\\
	Let $\alpha_{k+1}^{D_k} = \alpha_k^{D_k}$\\
	\ForAll{$v^d \in D_k$}{
		Let $\alpha_\text{ini} = \min\left\{\alpha_k^{v^d}, \max\limits_{\bar{\alpha} \in \mathbb{R}}\left\{\bar{\alpha}\mid z + \bar{\alpha}v^d \in \Omega_z\right\} \right\}$\label{line::startds}\\
		\If{$f\left(x, z+\alpha_\text{ini}v^d\right) \le f(x, z) - \xi_k$}{
			$\alpha_{k+1}^{v^d} = \max\limits_{h \in \mathbb{N}}\{2^h\alpha_\text{ini} \mid z + 2^h\alpha_\text{ini}v^d \in \Omega_z$ \textbf{and} $f\left(x, z + 2^h\alpha_\text{ini}v^d\right) \le f(x, z) - \xi_k\}$\label{line::alpha_d}\\
			$z^d = z + \alpha_{k+1}^{v^d}v^d$\label{line::endds}\\
			\Return $(x, z^d)$, $D_k$, $\alpha_{k+1}^{D_k}$, $\xi_k$\label{line::new_point}
		}
		\Else{
			$\alpha_{k+1}^{v^d} = \max\left\{1, \left \lfloor{\alpha_k^{v^d} / 2}\right \rfloor \right\}$\label{line::halved}
		}
	}
	\If{$\alpha_{k+1}^{D_k} \ne \mathbf{1}_{|D_k|}$}{
		\Return $(x, z)$, $D_k$, $\alpha_{k+1}^{D_k}$, $\xi_k$
	}
	\Else{
		$\xi_{k+1} = \eta\xi_k$\label{line::xi_reduction}\\
		\If{$D_k \subset \mathcal{D}^z(x, z)$}{
			Generate $D_{k+1}$ s.t. $D_k \subset D_{k+1} \subseteq \mathcal{D}^z(x, z)$\label{line::Denrich}\\
			Let $\alpha_{k+1}^{D_{k+1}} \in \mathbb{R}^{|D_{k+1}|}$ vector of items $\alpha^v = 1$ with $v \in D_{k+1}$\\
			\Return $(x, z)$, $D_{k+1}$, $\alpha_{k+1}^{D_{k+1}}$, $\xi_{k+1}$
		}
		\Else{
			\Return $(x, z)$, $D_k$, $\alpha_{k+1}^{D_k}$, $\xi_{k+1}$\label{line::D_full}
		}
	}
\end{algorithm}

The two key phases are based on line search algorithms exploring feasible search directions: the goal is to return a positive stepsize $\alpha$ which leads to a new point providing sufficient decrease of the objective function. If such a point can be determined, an expansion step is tried out to possibly obtain sufficient decrease with an even larger stepsize.  Otherwise, the initial tentative stepsize is reduced for the next iteration. 

For the continuous variables, these operations are expressed in steps \ref{line::expansion} and \ref{line::reduction}.
On the other hand, we report the detailed scheme of the \texttt{DS} phase in Algorithm \ref{alg::ds}, as it will be crucial for the rest of the paper. The discrete search (steps \ref{line::startds}-\ref{line::halved}) terminates as soon as a point providing sufficient decrease of the objective function is reached (step \ref{line::new_point}), or after all directions in $D_k$ have been processed. 
If the line search fails for some direction $v^d \in D_k$, its associated initial stepsize $\alpha_k^{v^d}$ is halved for the next \texttt{DS} execution (step \ref{line::halved}). 

Note that, for computational reasons, only a subset of the directions set $\mathcal{D}^z(x_k^c, z_k)$ is employed. 
Whenever the \texttt{DS} phase does not return a new point and each direction $v^d \in D_k$ is associated with a stepsize $\alpha_k^{v^d} = 1$, the sufficient decrease parameter $\xi_k$ is reduced and the set $D_k$ is enriched, if possible, with new primitive directions contained in $\mathcal{D}^z(x_k^c, z_k)$ (steps \ref{line::xi_reduction} and \ref{line::Denrich}).

For Algorithm \ref{alg::dfndfl}, some nice properties, including convergence results, have been proved.

\begin{lemma}[{\cite[Prop.\ 2]{giovannelli2022derivative}}]
	\label{lem::well-defined-d}
	Step \ref{line::alpha_d} of the Discrete Search phase is well-defined.
\end{lemma}

\begin{lemma}[{\cite[Prop.\ 4]{giovannelli2022derivative}}]
	\label{lem::xi}
	Let $\{\xi_k\}$ be the sequence produced by \texttt{DFNDFL}. Then, $\lim_{k \rightarrow \infty}\xi_k = 0$.
\end{lemma}

\begin{lemma}[{\cite[Theorem 1]{giovannelli2022derivative}}]
	\label{lem::conv_dfndfl}
	Let $\{(x_k, z_k)\}$ be the sequence of points generated by \texttt{DFNDFL}. Let $H \subseteq \{0, 1, 2,\ldots\}$ be defined as 
	\begin{equation*}
		H = \{k \in \mathbb{N} \mid \xi_{k + 1} < \xi_k\}
	\end{equation*}
	and let $\{v_k^c\}_{k \in H}$ be a dense subsequence in the unit sphere. Then,
	\begin{enumerate}[(i)]
		\item a limit point of $\{(x_k, z_k)\}_{k \in H}$ exists;
		\item every limit point $(x^\star, z^\star)$ of $\{(x_k, z_k)\}_{k \in H}$ is stationary for problem \eqref{eq:prob}.
	\end{enumerate}
\end{lemma}

\begin{remark}
	Since we assume in this paper that $f$ is continuously differentiable over $\Omega_x$, the Clarke–Jahn stationarity result from  \cite[Theorem 1]{giovannelli2022derivative} immediately translates into a standard stationarity result in Lemma \ref{lem::conv_dfndfl}. In fact, a convergent version of \texttt{DFNDFL} taking into account the continuous differentiability of $f$ with respect to $x$ could be devised by using a set of directions which positively spans $\mathbb{R}^n$, e.g., the set $\{\pm e_i,\ i=1,\dots,n\}$.
\end{remark}

\section{The Algorithmic Framework}
\label{sec:algo}

In this section, we describe our proposed approach, called \textit{Grad-DFL} (\texttt{G-DFL}), whose algorithmic scheme is reported in Algorithm \ref{alg:GradDFL}. The main difference w.r.t.\ \texttt{DFNDFL} lies in the optimization step of continuous variables: here, continuous local search steps are carried out by means of a first-order optimizer for bound-constrained nonlinear optimization (e.g., Projected Gradient method \cite{bertsekas1999nonlinear},  Frank-Wolfe algorithm \cite{frank1956algorithm} or \texttt{L-BFGS-B} \cite{byrd95}).

\begin{algorithm}[!h]
	\caption{Grad-DFL (\texttt{G-DFL})}
	\label{alg:GradDFL}
	\textbf{Input:} $(x_0, z_0) \in \Omega$,  $\xi_0 > 0$, $\gamma \in (0, 1)$, $\eta \in (0, 1)$\\
	Choose $D_0 =\{v_1,\ldots,v_{|D_0|}\} \subseteq \mathcal{D}^z(x_0, z_0)$\\
	Let $\alpha_0^{D_0} = [\alpha_0^{v_1},\ldots,\alpha_0^{v_{|D_0|}}] = \boldsymbol{1}_{|D_0|}$ \\
	Let $\alpha_0^c = 1$ \\
	\For{$k = 0, 1,\ldots$}{
		$\hat{z}_k^d, D_{k+1}, \alpha_{k+1}^{D_{k+1}}, \xi_{k+1}$ = \texttt{DS}$\left(f, \Omega, (x_k, z_k), D_k, \alpha_k^{D_k}, \xi_k, \eta\right)$\label{line::gds}\\
		\If{$\hat{z}_k^d\neq z_k$ \label{step:check_z}}{Find $z_k^d\in\Omega_z$ s.t.\ $f(x_k,{z}_k^d)\le f(x_k,\hat{z}_k^d)$}
		\Else{$z_k^d=z_k$ \label{step:else_z}}
		\If{$(x_k, z_k^d)$ is not stationary w.r.t.\ $x$\label{line::check_stat}}{
			Compute a feasible descent direction $v^c_k = \tilde{x}_k^c - x_k$, with $\tilde{x}_k^c \in \Omega_x$\label{line::dd}\\
			$\alpha_{k + 1}^c = \max\limits_{h \in \mathbb{N}}\{\delta^h\alpha_0^c\mid f(x_k + \alpha_{k+1}^cv^c_k, z_k^d) \le f(x_k, z_k^d) + \gamma\alpha_{k+1}^c\nabla_x f(x_k, z_k^d)^\top v^c_k$\} \label{line::alphagrad}\\
		}
		\Else{
			$v_k^c = \mathbf{0}_n$\\
			$\alpha_{k+1}^c = 0$
		}
		$x_k^c = x_k + \alpha_{k+1}^cv^c_k$\label{line::end_dd}\\
		Find $(x_{k+1}, z_{k+1}) \in \Omega$ s.t. $f(x_{k+1}, z_{k+1}) \le f(x_k^c, z_k^d)$ \label{line::new_xk+1}
	}
\end{algorithm}

At each iteration $k$ of \texttt{G-DFL}, we start moving the discrete variables through the \texttt{DS} phase (step \ref{line::gds}); if this step is successful, we are actually allowed to move to any other configuration of integer variables that further improves the objective value. Then, we perform a gradient-based minimization for the continuous ones (steps \ref{line::check_stat}-\ref{line::end_dd}). First, in step \ref{line::check_stat} we check if the current point $(x_k, z_k^d)$ is stationary w.r.t.\ the continuous variables; if it is not, we can then find a feasible descent direction at $x_k$ (step \ref{line::dd}). Given such direction, we employ an Armijo-type line-search \cite{bertsekas1999nonlinear} in step \ref{line::alphagrad}, so that the stepsize $\alpha_{k + 1}^c$ leads to a new feasible point satisfying the Armijo sufficient descent condition. Like in \texttt{DFNDFL}, step \ref{line::new_xk+1} allows the employment of additional minimization procedures to find the new iterate $(x_{k+1}, z_{k+1})$ from the point $(x_k^c, z_k^d)$.

\begin{remark}
	\label{remark:multiple_steps}
	Note that, in a single iteration of \texttt{G-DFL}, we can actually carry out multiple continuous local search steps without spoiling the convergence results of the algorithm presented in the following Section \ref{sec::convergence}. In fact, these additional steps can be encapsulated in step \ref{line::new_xk+1} of Algorithm \ref{alg:GradDFL}. 
	Similarly, steps \ref{step:check_z}-\ref{step:else_z} allow us to incrementally perform a sequence of \texttt{DS}-type steps, rather than a single one. 
\end{remark}

\begin{remark}
	Unlike \texttt{DFNDFL}, \texttt{G-DFL} performs the two minimization phases in inverted order. This is due to issues in the convergence analysis presented in the following section. From an experimental point of view, this structural change should have no relevant impact on the performance of the approach. We assessed this fact by some computational experiments that we do not report for the sake of brevity.
\end{remark}

\section{Convergence Analysis}
\label{sec::convergence}

In this section, we report the theoretical analysis for the \texttt{G-DFL} approach.
First, we can easily see that Lemmas \ref{lem::well-defined-d} and \ref{lem::xi} also hold for \texttt{G-DFL}. Then, the set of iteration indices
\begin{equation}\label{def:H_GDFL}
	H = \{k\in\mathbb{N}\ |\ \xi_{k+1} < \xi_k\} 
\end{equation}
is infinite.

Moreover, since in step \ref{line::dd} of the approach we employ a feasible descent direction $v_k^c$, we immediately get that standard line searches from the literature to find a stepsize $\alpha_{k+1}^c$ satisfying the Armijo rule are well-defined (see, e.g., \cite[Proposition 20.4]{grippo2023introduction}). Thus, putting this result together with Lemma \ref{lem::well-defined-d}, we have that the entire procedure is well-defined.

In the next proposition, we state the convergence property of \texttt{G-DFL} related to the discrete variables.

\begin{proposition}
	\label{prop::conv_int}
	Let $\{(x_k, z_k)\}$ and $\{\xi_k\}$ be the sequences produced by \texttt{G-DFL}. Let $H \subseteq \{0, 1, 2,\ldots\}$ be defined as in \eqref{def:H_GDFL} and $(x^\star, z^\star) \in \Omega$ be any accumulation point of $\{(x_k, z_k)\}_{k \in H}$. Then $$f(x^\star, z^\star) \le f(x^\star, z), \qquad \forall (x^\star, z) \in \mathcal{B}^z(x^\star, z^\star).$$
\end{proposition}
\begin{proof}
	Let $K \subseteq H$ be an index set such that
	\begin{equation*}
		\lim_{k \rightarrow \infty, k \in K}(x_k, z_k) = (x^\star, z^\star).
	\end{equation*}
	By Definition \ref{def::convergence},  the definition of $H$ and instructions \ref{step:check_z}-\ref{step:else_z} of the algorithm, we thus have that for $k \in K$ sufficiently large
	\begin{equation}
		\label{eq::z*zk}
		z_k = z_k^d = z^\star.
	\end{equation}
	
	Now, let us consider a point $(x^\star, \bar{z}) \in \mathcal{B}^z(x^\star, z^\star)$. By definition of discrete neighborhood, there exists a direction $\bar{d} \in \mathcal{D}^z(x^\star, z^\star)$ such that $\bar{z} = z^\star + \bar{d}$ and $(x^\star, \bar{z}) \in \Omega$.
	Considering this result together with \eqref{eq::z*zk}, we obtain that, for $k \in K$ sufficiently large, $\bar{z} = z^\star + \bar{d} = z_k + \bar{d} \in \Omega_z$. We can then deduce that, by definition of discrete neighborhood, $\bar{d} \in \mathcal{D}^z(x_k, z_k)$.
	Taking into account the definitions of $K$ and $\bar{d}$, we thus have that $f(x_k, z_k + \bar{d}) > f(x_k, z_k) - \xi_k$ for $k \in K$ sufficiently large. Taking the limits for $k \rightarrow \infty$, with $k \in K$, and recalling Lemma \ref{lem::xi}, we get the thesis.
\end{proof}

As for the convergence property related to the continuous variables, we will assume that the employed gradient-based search directions are \textit{gradient-related}, whose definition \cite[Sec.\ 2.2.1]{bertsekas1999nonlinear} is adapted below for the mixed-integer scenario.
\begin{definition}
	Let $\{(x_k, z_k), v_k\}$ be the sequence generated by a feasible direction method where $x_k \in \Omega_x$, $z_k \in \Omega_z$, $v_k \in \mathbb{R}^n$ are such that $\nabla_x f(x_k, z_k)^\top v_k < 0$ and the continuous variables are updated as $x_{k+1} = x_k + \alpha_kv_k$, with $\alpha_k > 0$. We say that the direction sequence $\{v_k\}$ is \textit{gradient-related} to $\{(x_k, z_k)\}$ if, for any subsequence $\{(x_k, z_k)\}_{k \in K}$ that converges to a nonstationary point w.r.t.\ the continuous variables, the corresponding subsequence	$\{v_k\}_{k \in K}$ is bounded and satisfies $$\limsup_{k \rightarrow \infty, k \in K}\nabla_xf(x_k, z_k)^\top v_k < 0.$$
\end{definition}
Note that this assumption, which at first glance could seem strong, is indeed quite reasonable, since many standard approaches for continuous bound-constrained optimization problems, including the Frank-Wolfe \cite[Section 2.2.2]{bertsekas1999nonlinear} and the Projected Gradient \cite[Proposition 2.3.1]{bertsekas1999nonlinear} methods, actually rely on directions of this type. In Appendix \ref{app:L-BFGS-B} we furthermore show that the employment of \texttt{L-BFGS-B} is also theoretically sound.

\begin{proposition}
	\label{prop::conv_cont_all}
	Let $\{(x_k, z_k)\}$, $\{(x_k, z_k^d)\}$ be the sequences of points produced by \texttt{G-DFL}. Let $(x^\star, z^\star) \in \Omega$ be any accumulation point of $\{(x_k, z_k^d)\}$.
	If the sequence $\{v^c_k\}$ is gradient-related to $\{(x_k, z_k^d)\}$, then 
	\begin{equation}
		\label{eq::stat_cont_2}
		\nabla_x f(x^\star, z^\star)^\top(x - x^\star) \ge 0, \quad \forall x \in \Omega_x.
	\end{equation}
\end{proposition}
\begin{proof}
	Let $K$ be an index set such that
	\begin{equation}
		\label{eq::lim_2}
		\lim_{k \rightarrow \infty, k \in K}(x_k, z_k^d) = (x^\star, z^\star).
	\end{equation}
	By Definition \ref{def::convergence}, we thus have there exists $\bar{k} \in K$ such that, for all $k \in K$ with $k \ge \bar{k}$,
	\begin{equation}
		\label{eq::zkzkdz*_2}
		z_k^d = z^\star.
	\end{equation}

	By the instructions of Algorithm \ref{alg:GradDFL}, we know that the entire sequence $\{f(x_k, z_k)\}$ is monotonically non-increasing and thus, it admits limit $\bar{f}$, which is finite given the compactness of $\Omega$. Hence, $f(x_k, z_k) - f(x_{k+1}, z_{k+1}) \rightarrow 0$. By steps \ref{line::gds}, \ref{line::alphagrad} and \ref{line::new_xk+1}, we get that, for all $k$,
	$$f(x_k, z_k) - f(x_{k+1}, z_{k+1}) \ge f(x_k, z_k^d) - f(x_k^c, z_k^d) \ge -\gamma\alpha_{k + 1}^c\nabla_xf(x_k, z_k^d)^\top v_k^c.$$ Then, taking the limits for $k \rightarrow \infty$, we obtain that
	\begin{equation}
		\label{eq:limit-gr}
		\alpha_{k + 1}^c\nabla_xf(x_k, z_k^d)^\top v_k^c \rightarrow 0.
	\end{equation}
	
	Now, let us suppose by contradiction that $(x^\star, z^\star)$ is such that equation \eqref{eq::stat_cont_2} does not hold, i.e., $(x^\star, z^\star)$ is not stationary w.r.t.\ the continuous variables. In particular, since $\{v_k^c\}$ is gradient-related to $\{(x_k,z_k^d)\}$, this means that
	\begin{equation}\label{eq:contradiction_2}
		\limsup_{k\to\infty,k\in K}\nabla_x f(x_k,z_k^d)^\top v_k^c < 0.
	\end{equation}
	Combining \eqref{eq:limit-gr} and \eqref{eq:contradiction_2}, it immediately follows that 
	\begin{equation}
		\label{eq::limit_alpha_2}
		\lim_{k \rightarrow \infty, k \in K}\alpha_{k + 1}^c = 0.
	\end{equation}

	Thus, we have that, for $k \in K$ with $k \ge \bar{k}$ sufficiently large, \eqref{eq::zkzkdz*_2} holds and $\alpha_{k + 1}^c < \alpha_{k + 1}^c/\delta < 1$. Hence, by definition of $\{v_k^c\}$ (step \ref{line::dd}) and of the Armijo rule, we must have that $$f\left(x_k + \frac{\alpha_{k + 1}^c}{\delta}v_k^c, z^\star\right) - f\left(x_k, z^\star\right) > \gamma\frac{\alpha_{k + 1}^c}{\delta}\nabla_xf\left(x_k, z^\star\right)^\top v_k^c.$$
	By using the Mean Value Theorem, we also get that $$f\left(x_k + \frac{\alpha_{k + 1}^c}{\delta}v_k^c, z^\star\right) - f\left(x_k, z^\star\right) = \frac{\alpha_{k + 1}^c}{\delta}\nabla_xf\left(x_k + t_k\frac{\alpha_{k + 1}^c}{\delta}v_k^c, z^\star\right)^\top v_k^c,$$ with $t_k \in (0, 1)$.
	Thus, putting together the two results, we obtain that 
	\begin{equation}
		\label{eq::before_the_limits_2}
		\nabla_xf\left(x_k + t_k\frac{\alpha_{k + 1}^c}{\delta}v_k^c, z^\star\right)^\top v_k^c > \gamma\nabla_xf\left(x_k, z^\star\right)^\top v_k^c.
	\end{equation}
	Since $\{v_k^c\}$ is gradient-related, $\{v_k^c\}_{k \in K}$ is bounded and, then, there exists a subsequence $\bar{K} \subseteq K$ such that $v_k^c \rightarrow \bar{v}^c$ for $k \rightarrow \infty$ with $k \in \bar{K}$. Thus, taking the limits in \eqref{eq::before_the_limits_2} for $k \rightarrow \infty$ with $k \in \bar{K}$, recalling the continuity of $\nabla_x f$, \eqref{eq::lim_2} and \eqref{eq::limit_alpha_2}, we obtain that $$(1 - \gamma)\nabla_xf\left(x^\star, z^\star\right)^\top\bar{v}^c \ge 0.$$ Since $1-\gamma > 0$ and $\{v_k^c\}$ is gradient-related, we have a contradiction with \eqref{eq:contradiction_2}. Thus, the thesis follows.
\end{proof}

\begin{proposition}
	\label{prop::conv_cont}
	Let $\{(x_k, z_k)\}$, $\{(x_k, z_k^d)\}$ be the sequences of points produced by \texttt{G-DFL}. Let $H \subseteq \{0,1,2,\ldots\}$ be defined as in \eqref{def:H_GDFL} and let $(x^\star, z^\star) \in \Omega$ be any accumulation point of $\{(x_k, z_k)\}_{k \in H}$. If the sequence $\{v^c_k\}$ is gradient-related to $\{(x_k,z_k^d)\}$, then 
	\begin{equation*}
		\nabla_x f(x^\star, z^\star)^\top(x - x^\star) \ge 0, \quad \forall x \in \Omega_x.
	\end{equation*}
\end{proposition}
\begin{proof}
	First of all we note that, by definition \eqref{def:H_GDFL} of $H$, it results  
	\[
	(x_k, z_k) = (x_k,z_k^d),\qquad\text{for all}\ {k \in H}. 
	\]
	Now, let $K \subseteq H$ be an index set such that
	\begin{equation*}
		\lim_{k \rightarrow \infty, k \in K}(x_k, z_k) = \lim_{k \rightarrow \infty, k \in K}(x_k, z_k^d) = (x^\star, z^\star).
	\end{equation*}
	The thesis then directly follows from Proposition \ref{prop::conv_cont_all}.
\end{proof}

We conclude the section with the main convergence result of the proposed algorithm.

\begin{proposition}
	Let $\{(x_k, z_k)\}$, $\{(x_k, z_k^d)\}$ be the sequences of points generated by \texttt{G-DFL}. Let $H \subseteq \{0, 1, 2,\ldots\}$ be defined as in \eqref{def:H_GDFL} and let the sequence $\{v_k^c\}$ be gradient-related to $\{(x_k,z_k^d)\}$. Then,
	\begin{enumerate}[(i)]
		\item a limit point of $\{(x_k, z_k)\}_{k \in H}$ exists;
		\item every limit point $(x^\star, z^\star)$ of $\{(x_k, z_k)\}_{k \in H}$ is stationary for problem \eqref{eq:prob}.
	\end{enumerate}
\end{proposition}
\begin{proof}
	As for point (i), since, by instructions of the algorithm, $(x_k, z_k) \in \Omega$ for all $k \in H$, with $\Omega$ being a compact set, the sequence $\{(x_k, z_k)\}_{k \in H}$ admits limit points. Then, point (ii) straightforwardly follows from Propositions \ref{prop::conv_int} and \ref{prop::conv_cont}.
\end{proof}

\section{Numerical Experiments}
\label{sec::exp}

In this section, we report the results of thorough computational experiments, by which we assess the potential of \texttt{G-DFL}. The code\footnote{The implementation code of the \texttt{G-DFL} algorithm can be found at \href{https://github.com/pierlumanzu/g_dfl}{github.com/pierlumanzu/g\_dfl}.} for the experiments was written in \texttt{Python3}. The experiments were run on a computer with the following characteristics: Ubuntu 22.04, Intel Xeon Processor E5-2430 v2 6 cores 2.50 GHz, 32 GB RAM.

Concerning \texttt{G-DFL}, we tested three different optimizers for the update of continuous variables: the limited memory BFGS Quasi-Newton method (\texttt{L-BFGS-B}) \cite{byrd95}, the Projected Gradient method (\texttt{PG}) \cite{bertsekas1999nonlinear} and Frank-Wolfe algorithm (\texttt{FW}) \cite{frank1956algorithm}. As underlined in Section \ref{sec::convergence} and Appendix \ref{app:L-BFGS-B}, the employment of all the three procedures within \texttt{G-DFL} is theoretically sound.

As for the parameters setting, for the Armijo-type line search used in \texttt{PG} and \texttt{FW}, the values were chosen based on some preliminary experiments, which we do not report here for the sake of brevity. We finally set $\gamma = 10^{-4}$ and $\delta = 0.5$. 

The local solvers have been used in two different ways: we tried both to carry out a single descent step or to perform $p$ consecutive updates proportional to the number $N=n+m$ of variables. The variants of \texttt{G-DFL} where multiple steps of local search are carried out at each iteration are denoted by a symbol ``+'' in the name. Moreover, based again on some preliminary experiments, we set $p = N/10$. Obviously, the local search is anyway stopped earlier if an approximately stationary point is reached; the stationarity approximation tolerance is set to $\varepsilon = 10^{-7}$.

We considered for comparison the \texttt{DFNDFL} algorithm \cite{giovannelli2022derivative}, to show that gradient-based steps can indeed be beneficial within the framework. We also considered  a popular branch-and-bound based approach for bound-constrained mixed-integer nonlinear optimization, namely \texttt{B-BB}, available in the open source library \texttt{Bonmin} \cite{BONAMI2008186}. Finally, we included in the comparisons the \texttt{MISQP} solver proposed in \cite{Exler2007}, which is the direct comparison as a solver that exploits derivatives of continuous variables without requiring the relaxation of integer ones; since here we only deal with bound constraints, the method simplifies to a Quasi-Newton type algorithm for mixed-integer optimization. 

For all solvers, we set the parameters as suggested in the corresponding reference papers and the available codes. For \texttt{DFNDFL},  Sobol \cite{SOBOL1976236} and Halton \cite{Halton1960} sequences are used to generate dense continuous  and primitive discrete directions, respectively. Note that, since the same discrete local search is employed both in \texttt{DFNDFL} and in \texttt{G-DFL}, the parameter settings related to this phase are the same in both algorithms. \texttt{B-BB} was run in an \texttt{AMPL} environment, which is in turn managed through the \texttt{amplpy} Python interface. The \texttt{MISQP} approach was implemented in Python, solving the quadratic subproblems by means of Gurobi \cite{gurobi}.

In addition to the standard setup for \texttt{DFNDFL}, we also tested a variant (referred to as \texttt{DFNDFL-C}) that only considers coordinate directions for the update of continuous variables; since we have assumed $f$ to be differentiable w.r.t.\ continuous variables, density with respect to the continuous variables is indeed no more strictly necessary to have convergence guarantees. 

Moreover, while standard \texttt{B-BB} provides us with baseline results from an exact solver working in analytical settings, we also had interest in measuring its performance in a scenario where variables $z$ are absolutely unrelaxable: in this case, their value within the branch-and-bound process shall be rounded up to integers before evaluating $f$; we refer to this setup by \texttt{B-BB-T}. 

Since all the approaches except for \texttt{MISQP} include some random operations, they were run 5 times with 5 different seeds for the pseudo-random number generator. The results for the analysis were then obtained by calculating the averages of the metrics values over the 5 runs. Each solver is given a budget of 2 minutes of computation before stopping; the choice of this stopping condition is due to the proper termination criteria of the algorithms being structurally too different. Yet, early termination was enabled for \texttt{G-DFL} and \texttt{DFNDFL} when hitting a maximum number of tested discrete directions, set equal to 300. Of course, all natural stopping conditions employed in \texttt{B-BB}, indicating that the solution cannot provably be improved anymore, have also been considered.

We ran the approaches on a collection of 19 problems, each of which can be tested with different numbers of variables $N$ and, thus, with various choices of $n$ and $m$ (Table \ref{tab::pc}); the total number of considered instances was thus 304. The problems are the following: \texttt{rastrigin, ackley}, \texttt{dixon-price} (also considered in \cite{giovannelli2022derivative} with the names \texttt{problem115, problem206, problem208}), \texttt{expquad, mccormck, qudlin, probpenl, sineali, nonscomp, explin}, \allowbreak \texttt{explin2, biggsb1, bdexp, cvxbqp1, ncvxbqp1, ncvxbqp2, ncvxbqp3, \allowbreak chenhark} and \texttt{pentdi}, selected from the \texttt{Cutest}\footnote{\href{https://github.com/ampl/global-optimization/tree/master/cute}{github.com/ampl/global-optimization/tree/master/cute}} collection \cite{Gould2015}. In most of the experiments, we configured the problems so as to have the ratio $m/N = 1/50$. We will specify in the following when a different ratio will be considered. 
Note that all the tested instances are originally bound constrained continuous problems; in order to transform each one into a mixed-integer problem, we added the integrity constraint for the last $m$ variables.

\begin{table}[h]
	\centering
	\footnotesize
	\renewcommand{\arraystretch}{1.3}
	\begin{tabular}{|c||c|}
		\hline
		$\mathbf{N}$&$\mathbf{m}$\\
		\hline
		\hline
		100&\textbf{2}, 5, 7, 10, 20, 40\\
		\hline
		200&$\mathbf{4}$\\
		\hline
		500&$\mathbf{10}$\\
		\hline
		1000&2, 5, 10, \textbf{20}, 50, 100\\
		\hline
		2000&\textbf{40}\\
		\hline
		5000&\textbf{100}\\
		\hline
	\end{tabular}
	\caption{Problems configurations. The values for $m$ in bold correspond to the 2\% of the total number of variables $N$ for the corresponding problem.}
	\label{tab::pc}
\end{table}

Algorithms have been compared based on the following metrics: the objective value of the returned solution $f^\star$, the runtime $T$, the number of iterations $N_{it}$ and the number of function (and derivatives) evaluations $N_{f}$. We also took into account the values of the latter three metrics attained at the moment the best solution is found; we denote these values by $T^{best}$, $N_{it}^{best}$ and $N_f^{best}$ respectively. To provide compact views of the results, we systematically make use of performance profiles \cite{Dolan2002}.
Unless stated otherwise, performance profiles of $T$, $N_{it}$, and $N_f$ have been computed considering only the problem instances where the involved solvers reached the same solution. 

To summarize results in terms of $f^\star$, we instead show the cumulative distribution of relative gap to the optimal value $\frac{|f-f^*|}{|f^*|}$, i.e., we plot the cumulative distribution function of the relative gap between the score achieved by a solver and the best score among those obtained by all the tested solvers. This kind of plot is more suitable than standard performance profiles when the considered metric can take both positive and negative values and does not indicate an absolute cost. 

\subsection{Preliminary Analysis of Grad-DFL}
Before comparing \texttt{G-DFL} with other approaches, an analysis is needed of the local optimizers it can be equipped with. For this preliminary stage, we considered a subset of the problems listed in Section \ref{sec::exp}: \texttt{rastrigin, ackley, dixon-price, expquad, mccormck, explin} and \texttt{ncvxbqp1}. In Figure \ref{fig::lbfgs}, we show the cumulative distributions of the relative gap to the best function value and the performance profiles w.r.t.\ $T$ for \texttt{G-DFL} equipped with    \texttt{PG}, \texttt{PG+}, \texttt{FW}, \texttt{FW+}, \texttt{L-BFGS-B} and \texttt{L-BFGS-B+}.

\begin{figure}[h]
	\centering
	\subfloat[Cumulative distribution of $\frac{f^\star-f_\text{best}}{|f_\text{best}|}$]{\includegraphics[width=0.45\textwidth]{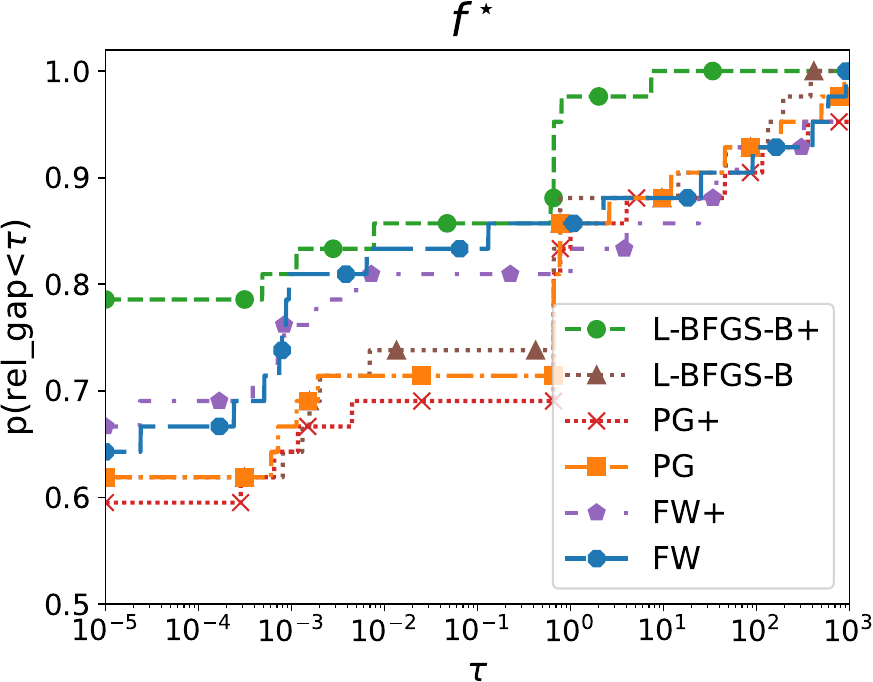}}
	\hfil
	\subfloat[Performance profiles w.r.t.\ $T$ on the 19 problems where all the solvers reached the same solution.]{\includegraphics[width=0.45\textwidth]{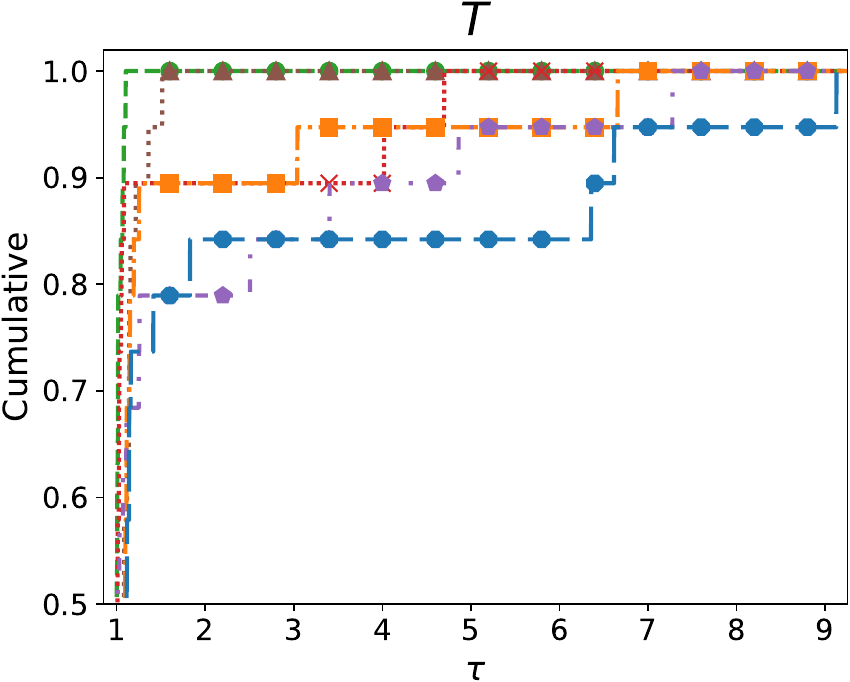}}
	\caption{Performance comparison for \texttt{G-DFL} equipped with different continuous local searches.
	}
	\label{fig::lbfgs}
\end{figure}

We observe that the employment of \texttt{L-BFGS-B+} led to far superior performance than any other variant in terms of $f^\star$, and it also appears to be the best option in terms of efficiency.  It is interesting to note that the performance of \texttt{PG} and \texttt{L-BFGS-B}, in terms of objective value, is similar; this shall not be surprising taking into account the discussion in Appendix \ref{app:L-BFGS-B}.

It is also worth noting that the execution of multiple continuous steps at each iteration is apparently crucial to boost the performance of \texttt{L-BFGS-B}, whereas the differences between \texttt{PG}/\texttt{PG+} and \texttt{FW}/\texttt{FW+} are less marked. This is arguably due to the fact that several iterations are needed for the limited-memory Quasi-Newton matrix to provide significant second order information.

The results of these preliminary experiments clearly lead us to use \texttt{L-BFGS-B+} as continuous local search in \texttt{G-DFL} for the rest of the work.

\subsection{Efficiency Evaluation}
In this section, we measure the efficiency of \texttt{G-DFL} upon the entire considered benchmark of problems, comparing it with \texttt{DFNDFL} and \texttt{B-BB}. 

Since in the following we will be comparing derivative-free and gradient-based approaches, it is crucial for getting intelligible results to know the ratio between the cost of computing $f$ and $\nabla_x f$. We noticed in fact that the usual assumption $\text{cost}(\nabla_x f)\approx n\times \text{cost}(f)$ represents a very rough approximation. 
Indeed, a more careful analysis, reported in Figure \ref{fig::boxplot} and Table \ref{tab::boxplot}, hints that the cost of evaluating gradients is often over-estimated. Here, we considered a subset of problems (\texttt{rastrigin, biggsb1, bdexp, mccormck, ncvxbqp2}) and for different values of $n$ we evaluated the computational time for calculating the gradient ($T_g$) and the objective function ($T_f$) at 500 randomly drawn feasible solutions.

\vspace{0.5cm}

\hskip-0.6cm
\begin{minipage}{\textwidth}
	\begin{minipage}{0.49\textwidth}
		\centering
		\includegraphics[width=0.7\textwidth]{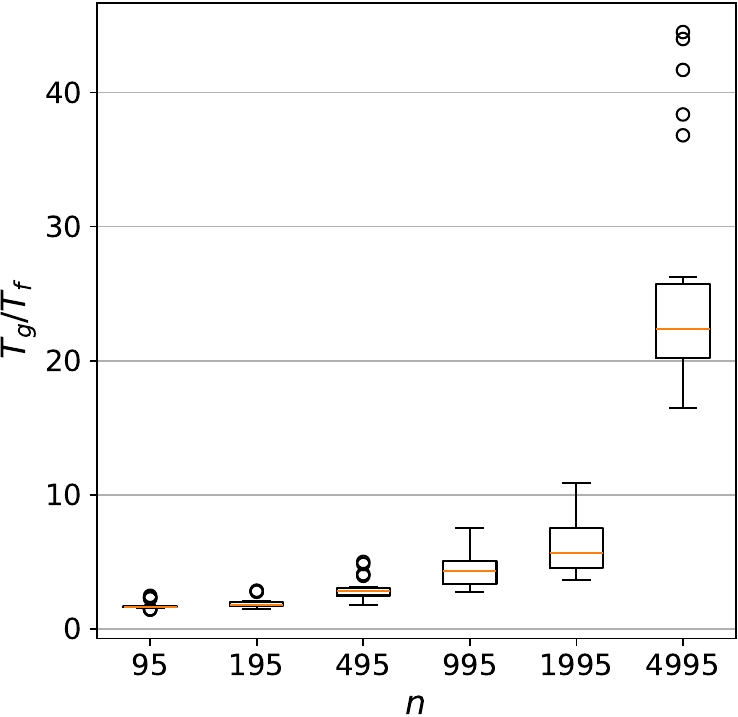}
		\captionof{figure}{Box plot of the ratio $T_g/T_f$; for each considered value of $n$, function values and gradients of 5 problems are computed at 500 feasible solutions.}
		\label{fig::boxplot}
	\end{minipage}
	\hfill
	\begin{minipage}{0.49\textwidth}
		\centering
		\scriptsize
		\renewcommand{\arraystretch}{1.3}
		\begin{tabular}{|c||c|c|}%
			\hline%
			\multirow{2}{*}{$\mathbf{n}$}&\multicolumn{2}{c|}{\textbf{$\mathbf{T_g}$\textbf{/}$\mathbf{T_f}$}}\\%
			\cline{2%
				-%
				3}%
			&\textit{Min}&\textit{Max}\\%
			\hline%
			\hline%
			95&1.444&2.47\\%
			\hline%
			195&1.492&2.852\\%
			\hline%
			495&1.781&4.999\\%
			\hline%
			995&2.771&7.543\\%
			\hline%
			1995&3.656&10.908\\%
			\hline%
			4995&16.441&44.513\\%
			\hline%
		\end{tabular}%
		\captionof{table}{Minimum and maximum values of the ratio $T_g/T_f$ for different problem sizes. For each considered value of $n$, function values and gradients of 5 problems are computed at 500 feasible solutions.}
		\label{tab::boxplot}
	\end{minipage}
\end{minipage}

\vspace{0.5cm}

In order to make the comparison in terms of $N_f$ and $N_f^{best}$ useful, we decided to weigh the gradient evaluations by a factor equal to the maximum cost ratio attained at the corresponding problem size (Table \ref{tab::boxplot}).

\begin{figure}[!h]
	\centering
	\subfloat{\includegraphics[width=0.32\textwidth]{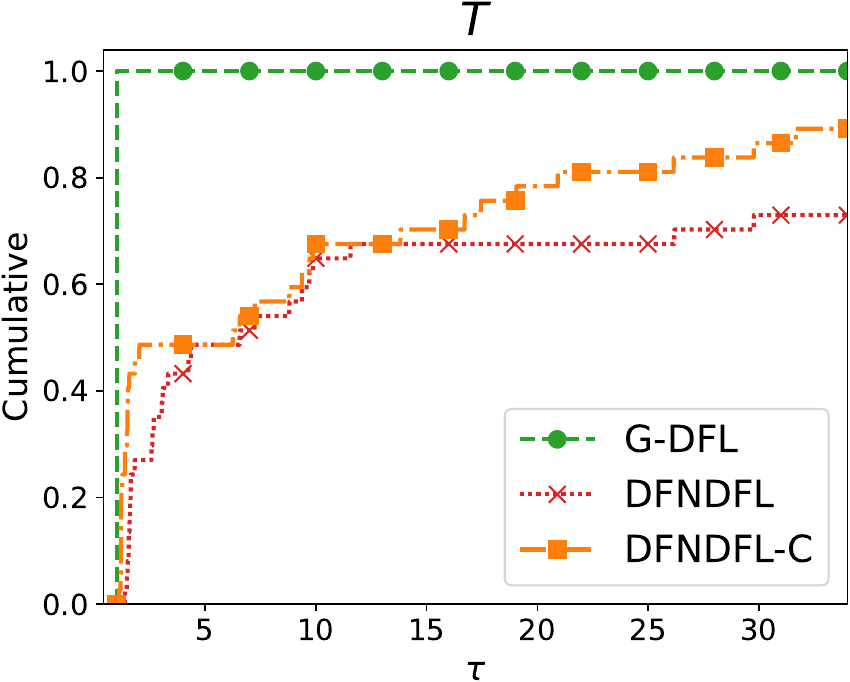}}
	\hfil
	\subfloat{\includegraphics[width=0.32\textwidth]{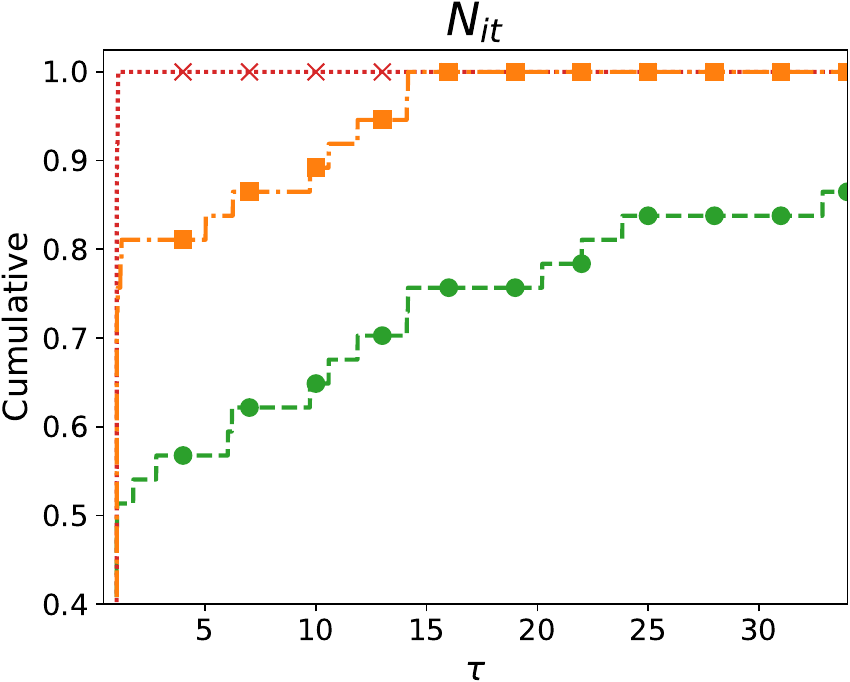}}
	\hfil
	\subfloat{\includegraphics[width=0.32\textwidth]{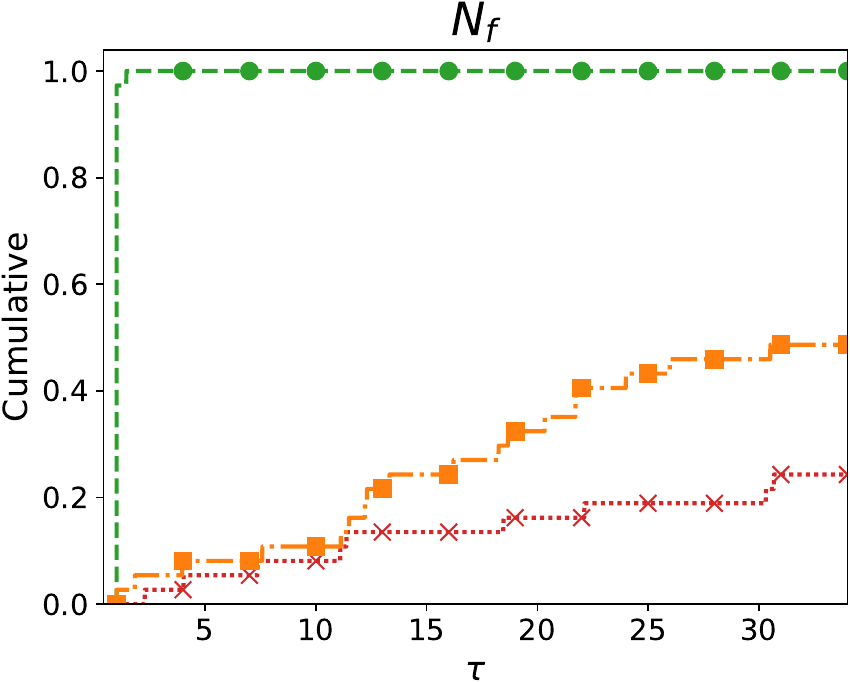}}
	\\
	\subfloat{\includegraphics[width=0.32\textwidth]{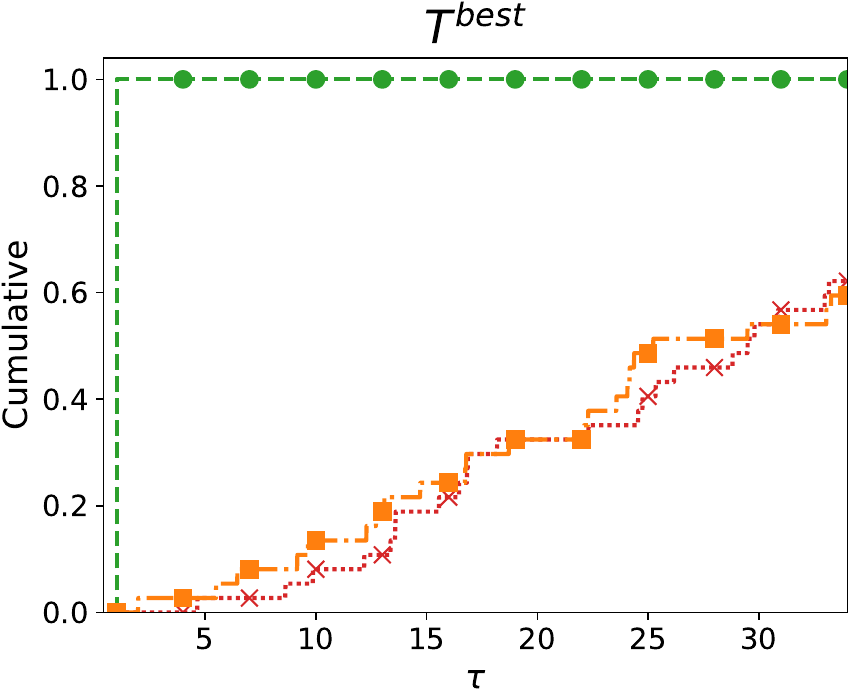}}
	\hfil
	\subfloat{\includegraphics[width=0.32\textwidth]{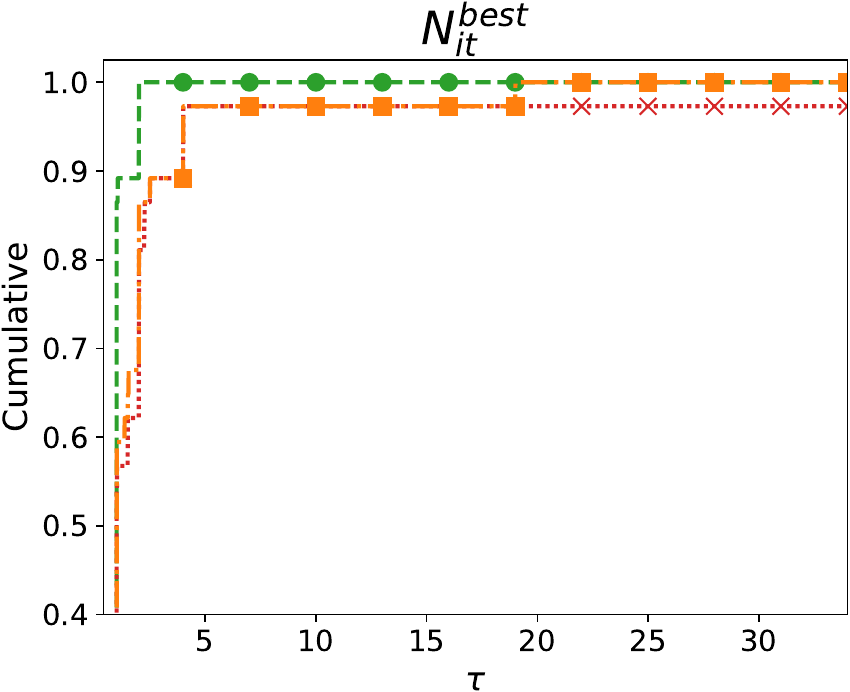}}
	\hfil
	\subfloat{\includegraphics[width=0.32\textwidth]{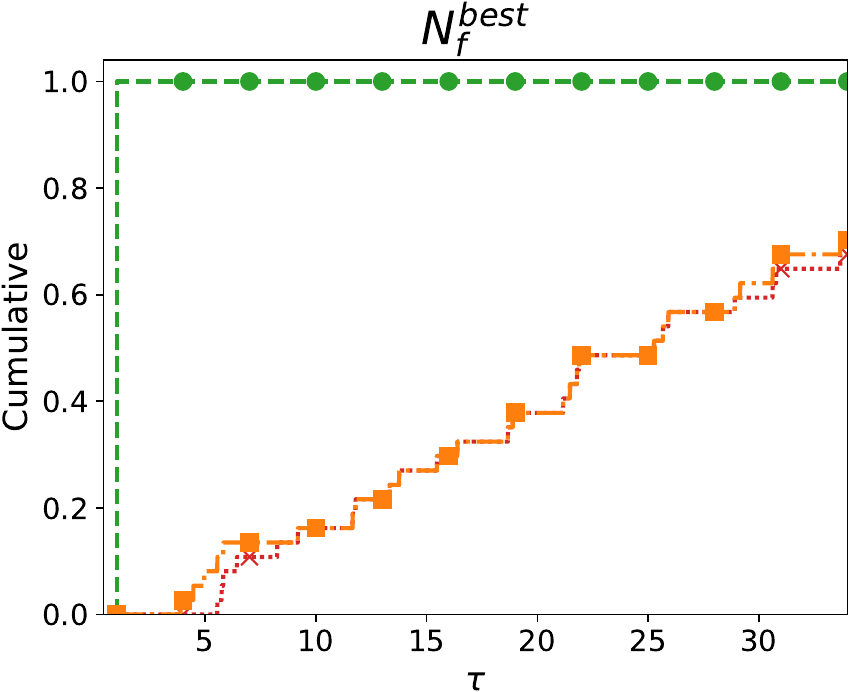}}
	\caption{Performance profiles for \texttt{G-DFL},
		\texttt{DFNDFL} and \texttt{DFNDFL-C}. The plots are made considering the 37 problems where the three solvers reached the same solution.}
	\label{fig::pp}
\end{figure}

In Figure \ref{fig::pp}, we report the performance profiles for \texttt{G-DFL}, \texttt{DFNDFL} and \texttt{DFNDFL-C}. We  observe that, maybe unsurprisingly, the results in terms of computation time are highly correlated to those about function evaluations. Indeed, \texttt{G-DFL} and \texttt{DFNDFL} share the overall structure and most basic operations, so that the difference lies in the end in the effort spent at evaluating functions.
\texttt{G-DFL} clearly outperformed its derivative-free counterpart. We might also note that, while \texttt{DFNDFL-C} performed slightly better than \texttt{DFNDFL} in terms of $T$ and $N_f$, the difference becomes negligible when we consider $T^{best}$ and $N^{best}_f$.

The performance profiles w.r.t.\ $N_{it}$ and $N_{it}^{best}$ highlight other computational aspects of the considered procedures. We see that \texttt{G-DFL} consistently carries out a larger number of iterations $N_{it}$; yet, being it also the fastest method, we deduce that \texttt{G-DFL} iterations are less computationally demanding than those of \texttt{DFNDFL}. Moreover, the greater computational cost required by an iteration of \texttt{DFNDFL} does not lead to the best solution with fewer iterations, as the plot w.r.t. $N_{it}^{best}$ attests.

Moving on to Figure \ref{fig::dc}, we plot the median runtime for \texttt{G-DFL}, \texttt{DFNDFL} and \texttt{DFNDFL-C} as the percentage of discrete variables $m$ grows on all problems with 100 total variables. 

\begin{figure}[!h]
	\centering
	\includegraphics[width=0.5\textwidth]{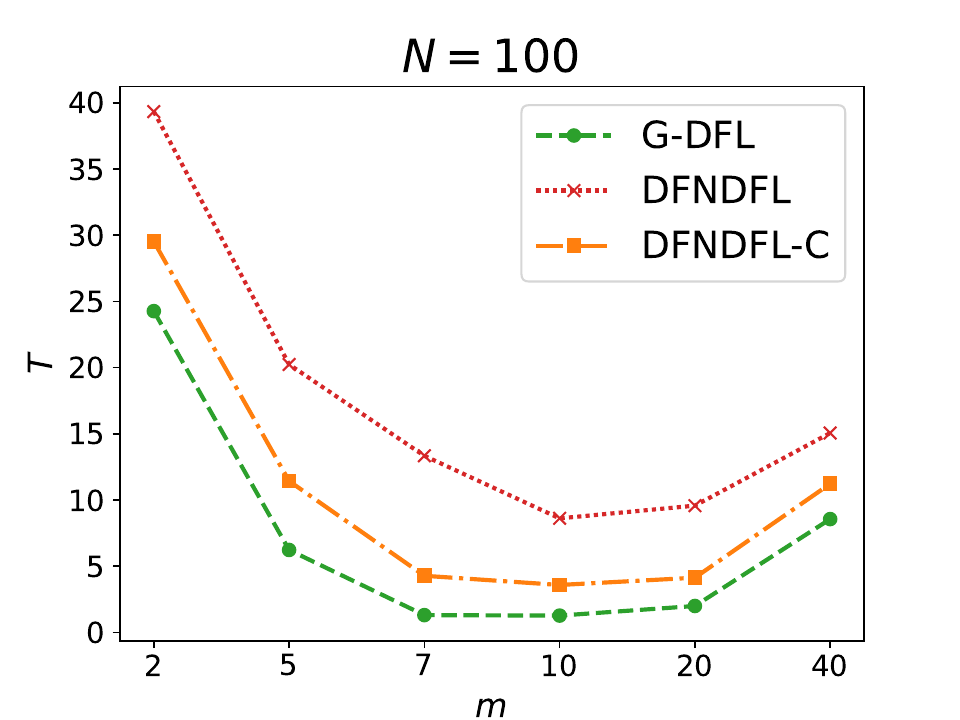}
	\caption{Plot of the median runtime for \texttt{G-DFL}, \texttt{DFNDFL} and \texttt{DFNDFL-C} for different values of $m$ on the 114 problems with $N=100$.}
	\label{fig::dc}
\end{figure}

\texttt{DFNDFL} and \texttt{DFNDFL-C} generally required higher runtimes than \texttt{G-DFL}. 
In the case of small $m$, we found the processing cost to be dominated by the quite high cost 
of computation of the 300 primitive discrete directions.
On the other hand, when $m$ is large, defining 300 primitive directions is trivial, but testing each one of them is computationally expensive. Problems with $m\sim10$ seem to be the least expensive to solve with DFL-type methods.

For the last part of this section, we included in the comparison the \texttt{B-BB} and the \texttt{MISQP} algorithms. 
In Figure \ref{fig::tfb_a}, we show the performance profiles w.r.t.\ $T^{best}$ for \texttt{G-DFL}, \texttt{DFNDFL}, \texttt{B-BB} and \texttt{MISQP}. The performance difference between \texttt{G-DFL} and \texttt{B-BB} appears small, whereas \texttt{DFNDFL} and \texttt{MISQP} seem to struggle. We have to point out that Figure \ref{fig::tfb_a} is based on the 19 problems where the three solvers reach the same solution. If, instead, all the benchmark problems are considered and an infinite value to the $T^{best}$ metric is set when a solver does not reach the optimum (Figure \ref{fig::tfb_b}), \texttt{G-DFL} becomes by far the best method.

\begin{figure}[!h]
	\centering
	\subfloat[Performance profiles on the 19 problems where the three solvers reach the same solution.\label{fig::tfb_a}]{\includegraphics[width=0.4\textwidth]{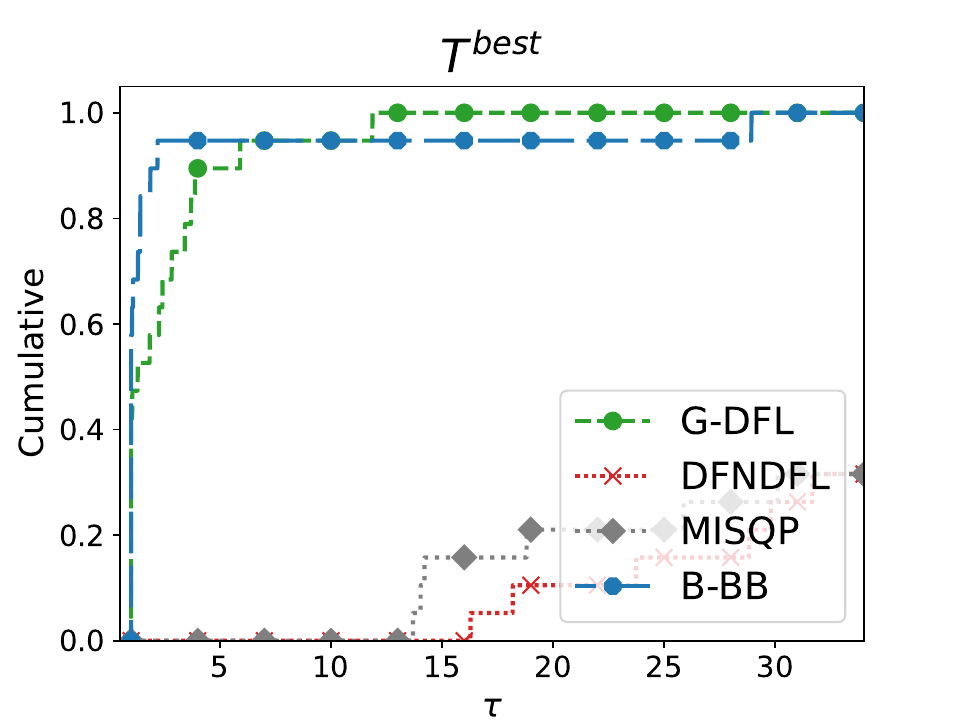}}
	\hfil
	\subfloat[Performance profiles where runtime is considered infinite whenever a solver does not reach the best solution for a specific problem.\label{fig::tfb_b}]{\includegraphics[width=0.4\textwidth]{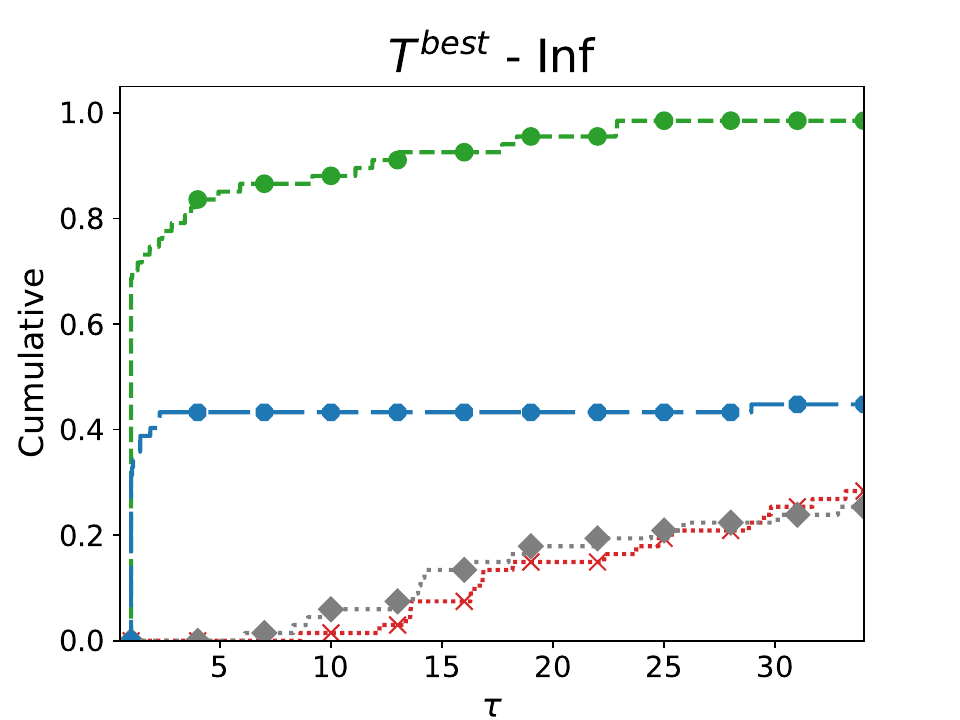}}
	\caption{Performance profiles w.r.t.\ $T^{best}$ for \texttt{G-DFL}, \texttt{DFNDFL} and \texttt{B-BB}.}
	\label{fig::tfb}
\end{figure}

\subsection{Effectiveness Analysis}
In this last section of the computational experiments, we carry out a detailed comparison in terms of the values of $f^\star$, reporting in Figures \ref{fig::obj_DFN}, \ref{fig::obj_Bonmin} and \ref{fig::obj_MISQP} the cumulative distributions of the relative gap to the best value for \texttt{G-DFL} compared to \texttt{DFNDFL}, \texttt{B-BB} and \texttt{B-BB-T}, and \texttt{MISQP}, respectively. In each case, we considered the results on all the benchmark problems and then the scenario restricted to the high-dimensional instances ($N > 500$). 

\begin{figure}[htbp]
	\centering
	\subfloat[Entire benchmark]{\includegraphics[width=0.4\textwidth]{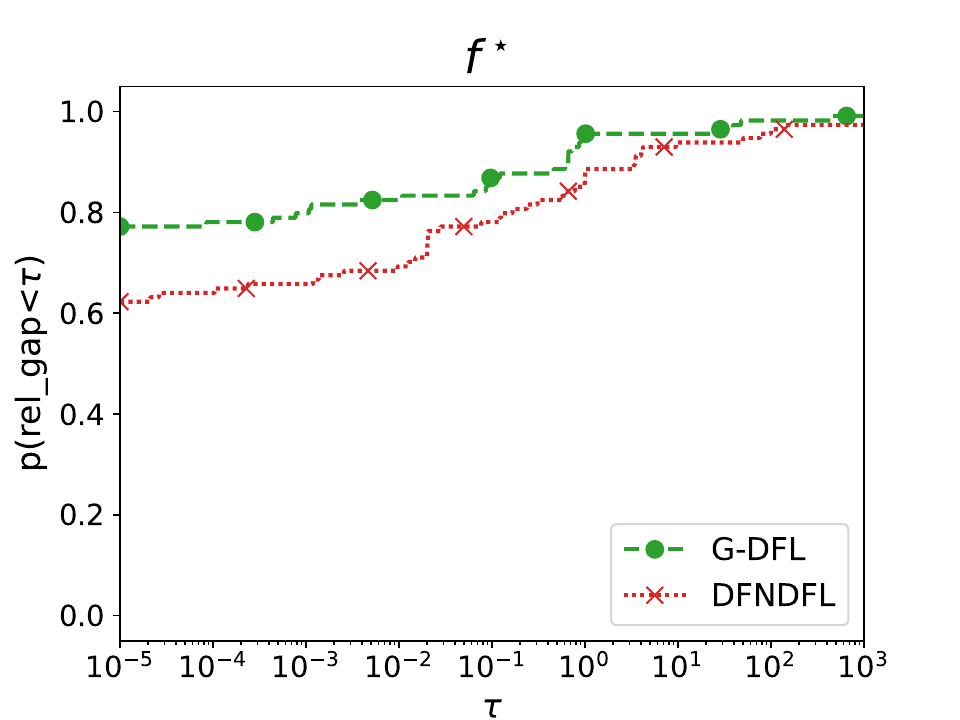}}
	\hfil
	\subfloat[$N > 500$]{\includegraphics[width=0.4\textwidth]{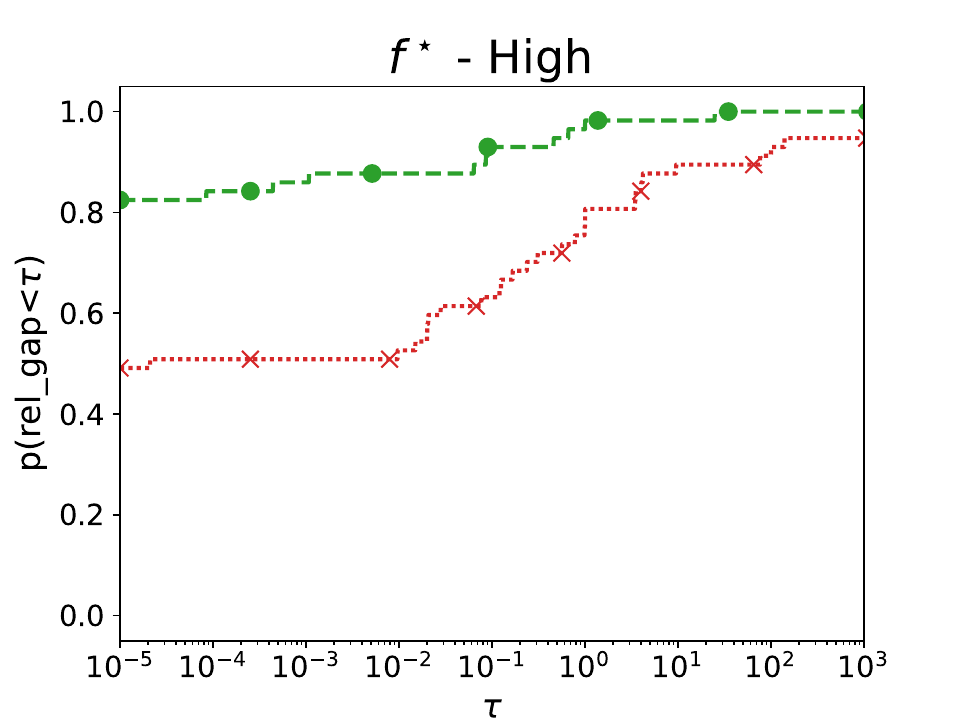}}
	\caption{Cumulative distributions of the relative gap to the best function value attained by \texttt{G-DFL} and \texttt{DFNDFL}.}
	\label{fig::obj_DFN}
\end{figure}

\begin{figure}[ht]
	\centering
	\subfloat[Entire benchmark]{\includegraphics[width=0.4\textwidth]{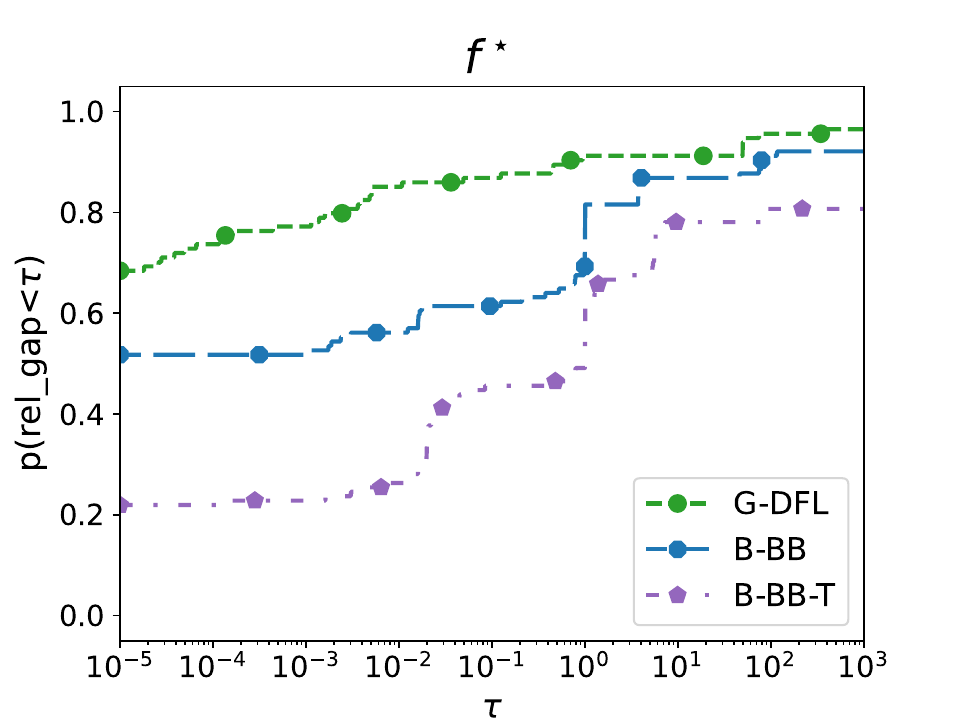}}
	\hfil
	\subfloat[$N > 500$]{\includegraphics[width=0.4\textwidth]{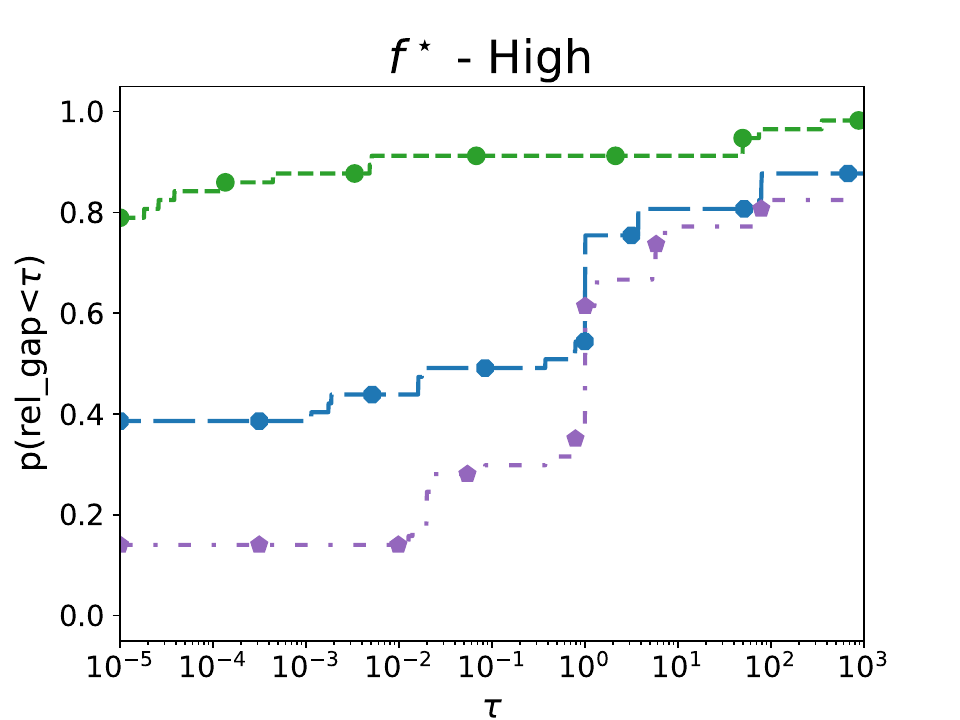}}
	\caption{Cumulative distributions of the relative gap to the best function value attained by \texttt{G-DFL}, \texttt{B-BB} and \texttt{B-BB-T}.}
	\label{fig::obj_Bonmin}
\end{figure}

\begin{figure}[ht]
	\centering
	\subfloat[Entire benchmark]{\includegraphics[width=0.4\textwidth]{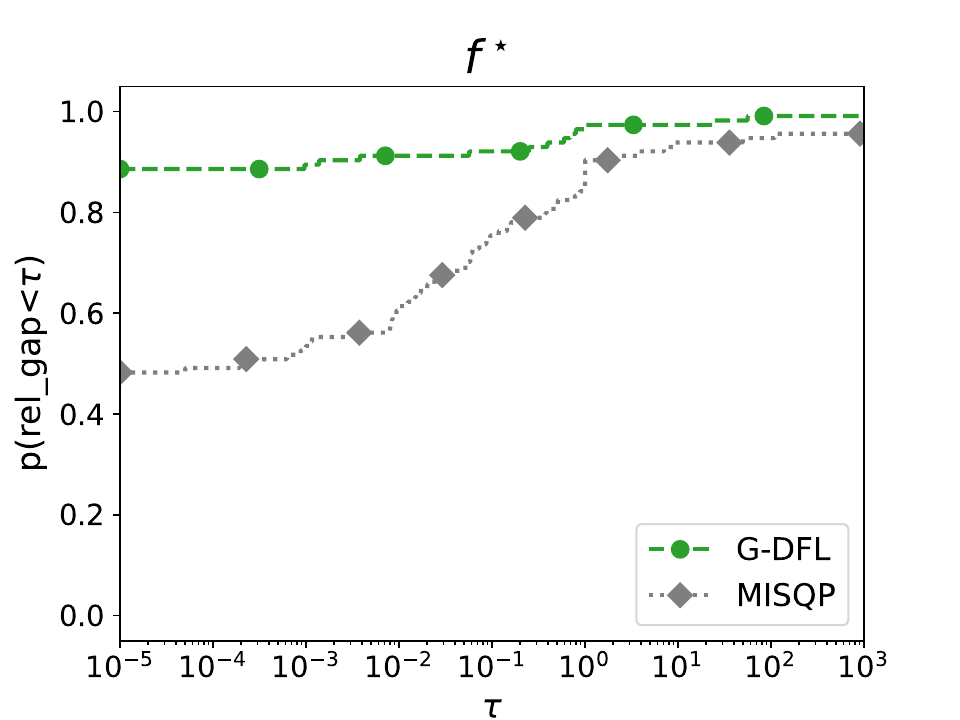}}
	\hfil
	\subfloat[$N > 500$]{\includegraphics[width=0.4\textwidth]{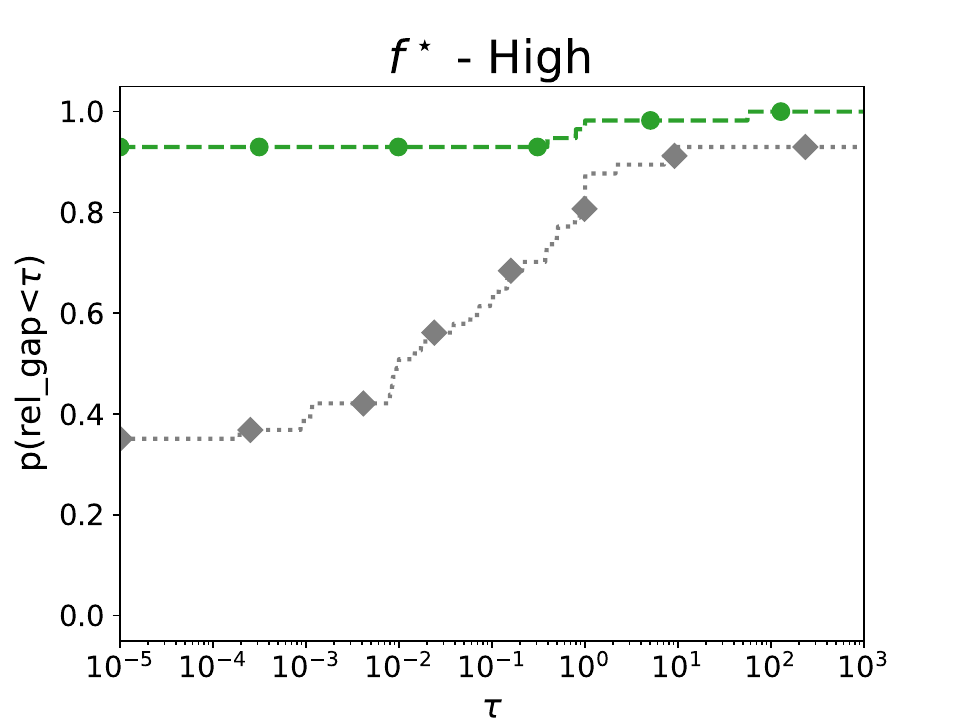}}
	\caption{Cumulative distributions of the relative gap to the best function value attained by \texttt{G-DFL} and \texttt{MISQP}.}
	\label{fig::obj_MISQP}
\end{figure}

Overall, \texttt{G-DFL} appears superior to all of the considered competitors. If we focus on cases where the number of variables becomes very large, the advantage becomes even more significant. It is particularly interesting to note that the performance of \texttt{B-BB} dramatically drops if relaxation of the integer variables is not possible (i.e, if we consider \texttt{B-BB-T}).

\section{Conclusions}
\label{sec:conclusion}
In this paper we proposed an algorithmic framework for bound-constrained mixed-integer nonlinear optimization. The distinctive characteristic of the method is that of exploiting, in an alternate fashion, descent steps along gradient-based directions to update continuous variables and steps along primitive directions to update integer variables. The algorithm can be used without any need of relaxing the integrality constraint throughout the process.

We provide for the method a solid theoretical analysis, showing that the sequence of produced solutions has at least a limit point satisfying a suitable necessary optimality condition. This property is shown to actually hold true when the most common solvers are employed for the update of continuous variables. 

Moreover, we carried out thorough computational experiments assessing i) the superiority of the method with respect to a derivative-free counterpart, encouraging the use of first-order information when possible, and ii) the good performance of the algorithm compared to popular mixed-integer nonlinear solvers, especially when the number of continuous variables is large or when integer variables are unrelaxable.

Future work might focus on obtaining stronger convergence properties for the algorithm without losing efficiency in practice. Moreover, the employment of second-order information on the continuous variables might also be considered.

\section*{Disclosure statement}

The authors report there are no competing interests to declare.

\section*{Notes on contributors}

\textbf{Matteo Lapucci} received his bachelor and master's degree in Computer Engineering at the University of Florence in 2015 and 2018 respectively. He then received in 2022 his PhD degree in Smart Computing jointly from the Universities of Florence, Pisa and Siena. He is currently Assistant Professor at the Department of Information Engineering of the University of Florence. His main research interests include theory and algorithms for sparse, multi-objective and large scale nonlinear optimization.
\\\\
\textbf{Giampaolo Liuzzi} is Associate Professor of operations research in the Department of Computer, Control, and Management Engineering at the University of Rome "La Sapienza." His research activities are primarily focused on the study of models for decision analysis and mathematical optimization, the study and development of methods for solving industrial problems, and the design of methods for machine learning.
\\\\
\textbf{Stefano Lucidi} received the M.S.\ degree in Electronic Engineering from the University of Rome ``La Sapienza'', Italy, in 1980. From 1982 to 1992 he was Researcher at the Istituto di Analisi e dei Sistemi e Informatica of the National Research Council of Italy. From September 1985 to May 1986 he was Honorary Fellow of the Mathematics Research Center of University of Wisconsin, Madison, USA. From 1992 to 2000 he was Associate Professor of Operations Research of University of Rome ``La Sapienza''. Since 2000 he is Full Professor of Operations Research of the University of Rome ``La Sapienza''. He teaches courses on Operations Research and Global Optimization methods in Engineering Management. He belongs to the Department of Computer, Control, and Management Engineering ``Antonio Ruberti'' of the University of Rome ``La Sapienza''. From November 2010 to October 2012, he was the Director of the PhD Programme in Operations Research of the University of Rome ``La Sapienza''. His research interests are mainly focused on the study, definition and application of nonlinear optimization methods and algorithms.
\\\\
\textbf{Pierluigi Mansueto} received his bachelor and master's degree in Computer Engineering at the University of Florence in 2017 and 2020, respectively. He then obtained his PhD degree in Information Engineering from the University of Florence in 2024. Currently, he is a Postdoctoral Research Fellow at the Department of Information Engineering of the University of Florence. His main research interests are multi-objective optimization and global optimization.

\section*{ORCID}

Matteo Lapucci: 0000-0002-2488-5486\\
Giampaolo Liuzzi: 0000-0002-4063-8370\\
Stefano Lucidi: 0000-0003-4356-7958\\
Pierluigi Mansueto: 0000-0002-1394-0937

\bibliographystyle{abbrv}

\begin{thebibliography}{10}
	
	\bibitem{Arroyo2009}
	S.~F. Arroyo, E.~J. Cramer, J.~E. Dennis, and P.~D. Frank.
	\newblock Comparing problem formulations for coupled sets of components.
	\newblock {\em Optimization and Engineering}, 10(4):557--573, Dec 2009.
	
	\bibitem{Audet2017}
	C.~Audet and W.~Hare.
	\newblock {\em Introduction: Tools and Challenges in Derivative-Free and
		Blackbox Optimization}, chapter~1, pages 3--14.
	\newblock Springer International Publishing, Cham, 2017.
	
	\bibitem{bertsekas1999nonlinear}
	D.~P. Bertsekas.
	\newblock {\em Nonlinear programming}.
	\newblock Athena Scientific, Belmont, MA, 2 edition, 1999.
	
	\bibitem{BONAMI2008186}
	P.~Bonami, L.~T. Biegler, A.~R. Conn, G.~Cornuéjols, I.~E. Grossmann, C.~D.
	Laird, J.~Lee, A.~Lodi, F.~Margot, N.~Sawaya, and A.~Wächter.
	\newblock An algorithmic framework for convex mixed integer nonlinear programs.
	\newblock {\em Discrete Optimization}, 5(2):186--204, 2008.
	\newblock In Memory of George B. Dantzig.
	
	\bibitem{BOUKOUVALA2016701}
	F.~Boukouvala, R.~Misener, and C.~A. Floudas.
	\newblock Global optimization advances in mixed-integer nonlinear programming,
	{MINLP}, and constrained derivative-free optimization, {CDFO}.
	\newblock {\em European Journal of Operational Research}, 252(3):701--727,
	2016.
	
	\bibitem{byrd95}
	R.~H. Byrd, P.~Lu, J.~Nocedal, and C.~Zhu.
	\newblock A limited memory algorithm for bound constrained optimization.
	\newblock {\em SIAM Journal on Scientific Computing}, 16(5):1190--1208, 1995.
	
	\bibitem{conn2009introduction}
	A.~R. Conn, K.~Scheinberg, and L.~N. Vicente.
	\newblock {\em Introduction to derivative-free optimization}.
	\newblock SIAM, 2009.
	
	\bibitem{cristofari2024mixalm}
	A.~Cristofari, G.~{Di Pillo}, G.~Liuzzi, and S.~Lucidi.
	\newblock An augmented lagrangian-based method using primitive directions for
	mixed-integer nonlinear problems.
	\newblock {\em Submitted to Computational Optimization and Applications}, 2024.
	
	\bibitem{Dolan2002}
	E.~D. Dolan and J.~J. Mor{\'e}.
	\newblock Benchmarking optimization software with performance profiles.
	\newblock {\em Mathematical Programming}, 91(2):201--213, Jan 2002.
	
	\bibitem{Exler2007}
	O.~Exler and K.~Schittkowski.
	\newblock A trust region sqp algorithm for mixed-integer nonlinear programming.
	\newblock {\em Optimization Letters}, 1(3):269--280, Jun 2007.
	
	\bibitem{Facchinei10}
	F.~Facchinei and C.~Kanzow.
	\newblock Penalty methods for the solution of generalized nash equilibrium
	problems.
	\newblock {\em SIAM Journal on Optimization}, 20(5):2228--2253, 2010.
	
	\bibitem{frank1956algorithm}
	M.~Frank, P.~Wolfe, et~al.
	\newblock An algorithm for quadratic programming.
	\newblock {\em Naval research logistics quarterly}, 3(1-2):95--110, 1956.
	
	\bibitem{giovannelli2022derivative}
	T.~Giovannelli, G.~Liuzzi, S.~Lucidi, and F.~Rinaldi.
	\newblock Derivative-free methods for mixed-integer nonsmooth constrained
	optimization.
	\newblock {\em Computational Optimization and Applications}, 82(2):293--327,
	2022.
	
	\bibitem{Gould2015}
	N.~I.~M. Gould, D.~Orban, and P.~L. Toint.
	\newblock Cutest: a constrained and unconstrained testing environment with safe
	threads for mathematical optimization.
	\newblock {\em Computational Optimization and Applications}, 60(3):545--557,
	Apr 2015.
	
	\bibitem{grippo2023introduction}
	L.~Grippo and M.~Sciandrone.
	\newblock {\em Introduction to Methods for Nonlinear Optimization}, volume 152.
	\newblock Springer Nature, 2023.
	
	\bibitem{gurobi}
	{Gurobi Optimization, LLC}.
	\newblock {Gurobi Optimizer Reference Manual}, 2024.
	
	\bibitem{Halton1960}
	J.~H. Halton.
	\newblock On the efficiency of certain quasi-random sequences of points in
	evaluating multi-dimensional integrals.
	\newblock {\em Numerische Mathematik}, 2(1):84--90, Dec 1960.
	
	\bibitem{Lapucci2023}
	M.~Lapucci, T.~Levato, F.~Rinaldi, and M.~Sciandrone.
	\newblock A unifying framework for sparsity-constrained optimization.
	\newblock {\em Journal of Optimization Theory and Applications},
	199(2):663--692, Nov 2023.
	
	\bibitem{larson2019derivative}
	J.~Larson, M.~Menickelly, and S.~M. Wild.
	\newblock Derivative-free optimization methods.
	\newblock {\em Acta Numerica}, 28:287--404, 2019.
	
	\bibitem{le2023taxonomy}
	S.~Le~Digabel and S.~M. Wild.
	\newblock A taxonomy of constraints in black-box simulation-based optimization.
	\newblock {\em Optimization and Engineering}, pages 1--19, 2023.
	
	\bibitem{liuzzi2020algorithmic}
	G.~Liuzzi, S.~Lucidi, and F.~Rinaldi.
	\newblock An algorithmic framework based on primitive directions and
	nonmonotone line searches for black-box optimization problems with integer
	variables.
	\newblock {\em Mathematical Programming Computation}, 12:673--702, 2020.
	
	\bibitem{liuzzi2012decomposition}
	G.~Liuzzi and A.~Risi.
	\newblock A decomposition algorithm for unconstrained optimization problems
	with partial derivative information.
	\newblock {\em Optimization Letters}, 6:437--450, 2012.
	
	\bibitem{Pedersen01}
	A.~Pedersen.
	\newblock Likelihood inference by monte carlo methods for incompletely
	discretely observed diffusion processes.
	\newblock {\em Research Report}, 2001-1:21, 2001.
	
	\bibitem{Sobieszczanski-Sobieski1997}
	J.~Sobieszczanski-Sobieski and R.~T. Haftka.
	\newblock Multidisciplinary aerospace design optimization: survey of recent
	developments.
	\newblock {\em Structural optimization}, 14(1):1--23, Aug 1997.
	
	\bibitem{SOBOL1976236}
	I.~Sobol.
	\newblock Uniformly distributed sequences with an additional uniform property.
	\newblock {\em USSR Computational Mathematics and Mathematical Physics},
	16(5):236--242, 1976.
	
\end{thebibliography}

\appendix
\section{L-BFGS-B first step is gradient-related}
\label{app:L-BFGS-B}
We show in this section that the assumptions for convergence results of \texttt{G-DFL} are indeed satisfied if \texttt{L-BFGS-B} is employed as optimizer in the continuous variables update step.

As noted in Remark \ref{remark:multiple_steps}, once a gradient-related step is carried out at each iterations, any additional optimization of continuous variables can be conceptually embedded within step \ref{line::new_xk+1} of Algorithm \ref{alg:GradDFL}, as long as the objective value is nonincreasing.
Therefore, focusing on the first iteration of \texttt{L-BFGS-B} can be sufficient to study the convergence.

Actually, \texttt{L-BFGS-B} employs a Wolfe line search at each iteration; the Armijo condition is thus certainly satisfied. It remains to understand whether the first direction is gradient-related. We show hereafter that the direction actually coincides with that of the gradient projection method, which is well-known to be gradient-related.

\texttt{L-BFGS-B} makes a sequence of operations to define the search direction, that are summarized here below:
\begin{enumerate}[(i)]
	\item given a positive definite matrix $B_k$, define the \textit{quadratic model} $$m_k(x) = f(x_k) + \nabla f(x_k)^\top (x-x_k) + \frac{1}{2}(x-x_k)^\top B_k(x-x_k);$$
	\item let $x(t) = \text{Proj}_{\Omega_x}[x_k-t\nabla f(x_k)]$ be the \textit{piece-wise linear path} in $\mathbb{R}^n$ obtained projecting $x_k-t\nabla f(x_k)$ onto the box $\Omega_x$ for values of $t\in[0,\infty)$; 
	\item find the \textit{first local minimizer} of $m_k(x)$ along the path $x(t)$; we thus obtain the \textit{Cauchy point} $$x_c^k = \text{Proj}_{\Omega_x}[x_k-t^\star\nabla f(x_k)],$$ where
	$$t^\star = \min\{t\ge 0\mid t \text{ is a local minimizer of } m_k(x(t))\};$$ 
	\item fixing the active constraints and ignoring the other constraints, (approximately) solve
	\begin{equation}
		\label{eq:subprob_lbfgsb}
		\begin{aligned}
			\min_{x}\;&m_k(x)
			\\ \text{s.t. }& x_i = (x_c^k)_i \text{ for all } i: (x_c^k)_i=l_i \text{ or } (x_c^k)_i=u_i;
		\end{aligned}
	\end{equation}
	the resulting point $\tilde{x}_k$ can be obtained by an iterative algorithm starting at $x_c^k$ and must satisfy $m_k(\tilde{x}_k)\le m_k(x_c^k)$;
	\item the path from $x_c^k$ to $\tilde{x}_k$ is truncated so as to obtain a feasible solution $\hat{x}_k$;
	\item the search direction is given by $\hat{x}_k-x_k$.
\end{enumerate}

We now consider the special case of the first iteration of \texttt{L-BFGS-B}; the peculiarity of this iteration is that $B_k=\gamma I$, $\gamma>0$; the quadratic model is therefore given by
$m_0(x) = f(x_0)+\nabla f(x_0)^\top (x-x_0) + \frac{\gamma}{2}\|x-x_0\|^2$. Let us now consider the problem $\min_{x\in\Omega_x} m_0(x)$. We have
\begin{align*}
	\arg\min_{x\in\Omega_x} \; m_0(x) &= \arg\min_{x\in\Omega_x} \; f(x_0)+\nabla f(x_0)^\top (x-x_0) + \frac{\gamma}{2}\|x-x_0\|^2\\
	& =\arg\min_{x\in\Omega_x}\;  \frac{2}{\gamma} \nabla f(x_0)^\top (x-x_0) + \|x-x_0\|^2\\
	& = \arg\min_{x\in\Omega_x}\; \frac{1}{\gamma^2}\|\nabla f(x_0)\|^2 + \frac{2}{\gamma} \nabla f(x_0)^\top (x-x_0) + \|x-x_0\|^2\\
	& = \arg\min_{x\in\Omega_x} \|x-x_0+\frac{1}{\gamma}\nabla f(x_0)\|^2\\& = \text{Proj}_{\Omega_x}[x_0-\frac{1}{\gamma}\nabla f(x_0)].
\end{align*}

We deduce that the (unique) global minimizer of the quadratic model $m_0(x)$ on the entire feasible region $\Omega_x$ actually lies in the piece-wise linear path $x(t)$ and is obtained for $t=\frac{1}{\gamma}$. We point out that this point is actually the first local minimizer of $m_0(x)$ along $x(t)$.

The path $x(t)$ is characterized by $q \le n$ breakpoints $0<t_1<\ldots<t_q$;  within each interval $(t_j, t_{j+1})$ function $p(x(t)) = \|x(t)-x_0+\frac{1}{\gamma}\nabla f(x_0)\|^2$, which is equivalent to $m_0(x(t))$ in optimization terms as shown above, is continuously differentiable; moreover, the values of a subset of variables $\mathcal{I}_j\subseteq \{1,\ldots,n\}$ are constantly fixed to the bounds, whereas for all $i\in \bar{\mathcal{I}}_j = \{1,\ldots,n\}\setminus \mathcal{I}_j$ we have $(x(t))_i = (x_0)_i-t\nabla_i f(x_0)$. We thus have
\begin{align*}
	\frac{\partial p(x(t))}{\partial t} &= \frac{\partial}{\partial t} \sum_{i=1}^{n}\left((x(t))_i-(x_0)_i+\frac{1}{\gamma}\nabla_i f(x_0)\right)^2\\&=\frac{\partial}{\partial t}\sum_{i\in \bar{\mathcal{I}}_j}\left((x_0)_i-t\nabla_i f(x_0)-(x_0)_i+\frac{1}{\gamma}\nabla_i f(x_0)\right)^2\\&=\sum_{i\in \bar{\mathcal{I}}_j}\frac{\partial}{\partial t} \left(\nabla_i f(x_0)\left(\frac{1}{\gamma}-t\right)\right)^2\\& = -2\left(\frac{1}{\gamma}-t\right)\sum_{i\in \bar{\mathcal{I}}_j}\nabla_if(x_0)^2.
\end{align*}
Since function $p(x(t))$ is convex in $(t_j,t_{j+1})$, the necessary and sufficient condition of optimality for $t$ in this interval is
$$0=\frac{\partial p(x(t))}{\partial t}=-2\left(\frac{1}{\gamma}-t\right)\sum_{i\in \bar{\mathcal{I}}_j}\nabla_if(x_0)^2,$$
i.e.,
$$t = \frac{1}{\gamma},$$
which is the unique globally optimal solution.
Thus, local, yet not global minimizers might only be found at the breakpoints.
Yet, we shall outline that the function is continuous at the breakpoints; moreover, even if the derivative does not exist at $t_j$, we have
\begin{gather*}
	\lim_{t\to t_j^-} \frac{\partial p(x(t))}{\partial t} = -2\left(\frac{1}{\gamma}-t_{j}\right)\sum_{i\in \bar{\mathcal{I}}_{j-1}}\nabla_if(x_0)^2,\\
	\lim_{t\to t_j^+} \frac{\partial p(x(t))}{\partial t} = -2\left(\frac{1}{\gamma}-t_{j}\right)\sum_{i\in \bar{\mathcal{I}}_j}\nabla_if(x_0)^2.  
\end{gather*}
At a local minimizer, the first limit cannot be strictly positive and the second limit cannot be strictly negative. Therefore, one of the following cases would hold:
\begin{itemize}
	\item $t_j = \frac{1}{\gamma}$, i.e., $t_j$ would again be the global minimizer;
	\item $t_j>\frac{1}{\gamma}$, i.e., the global minimizer appears along the path $x(t)$ before the considered local minimizer $t_j$;
	\item $t_j<\frac{1}{\gamma}$, with $$\sum_{i\in \bar{\mathcal{I}}_{j-1}}\nabla_if(x_0)^2\ge 0 \quad\text{ and }\quad \sum_{i\in \bar{\mathcal{I}}_{j}}\nabla_if(x_0)^2\le 0;$$ the second condition implies $\nabla_i f(x_0) = 0$ for all $i\in \bar{\mathcal{I}}_j$; but then, by definition of $x(t)$ and noting that $\bar{\mathcal{I}}_j \supset \bar{\mathcal{I}}_{h}$ for all $h>j$, we have $x(t) = x(t_j)$ for all $t>t_j$. Hence, $x(t_j) = x(\frac{1}{\gamma})$; the solution thus coincides once again with the global minimizer.
\end{itemize}

We have finally got that the Cauchy point will surely given by the global optimizer of the model on the entire feasible region and has the form $x_c^k = \text{Proj}_{\Omega_x}[x_0-\frac{1}{\gamma}\nabla f(x_0)]$; furthermore, $x_c^k$ is optimal for subproblem \eqref{eq:subprob_lbfgsb} and it is a feasible solution; thus, we have $x_c^k = \tilde{x}_k = \hat{x}_k$.

We have finally found that the direction at the first iteration  of \texttt{L-BFGS-B} is given by
$$d_0 = \text{Proj}_{\Omega_x}\left[x_0-\frac{1}{\gamma}\nabla f(x_0)\right]-x_0,$$
i.e., it is a gradient projection type direction, which is gradient-related.

\end{document}